\documentclass[a4paper]{siamart1116}
\usepackage{mathtools}
\usepackage{epstopdf}
\usepackage{algorithm,algorithmic}
\usepackage{amsmath}
\usepackage{amssymb}
\graphicspath{{images_color/}}
\usepackage{hyperref}
\usepackage[square,numbers]{natbib}
\usepackage{extdash}
\title{Approximate multiplication of nearly sparse matrices with decay in a fully recursive distributed task-based parallel framework%
\thanks{\today
}}
\author{Anton G. Artemov%
\thanks{Division of Scientific Computing, Department of Information Technology, Uppsala University, Box 337, SE-751 05 Uppsala, Sweden (\email{anton.artemov@it.uu.se}).}
}

\begin{document}
\maketitle

\begin{abstract}
In this paper we consider parallel implementations of approximate multiplication of large matrices with exponential decay of elements. Such matrices arise in computations related to electronic structure calculations and some other fields of computational science. Commonly, sparsity is introduced by dropping out small entries (truncation) of input matrices. Another approach, the sparse approximate multiplication algorithm [M. Challacombe and N. Bock, arXiv preprint 1011.3534, 2010] performs truncation of sub-matrix products. We consider these two methods and their combination, i.e. truncation of both input matrices and sub-matrix products. Implementations done using the Chunks and Tasks programming model and library [E. H. Rubensson and E. Rudberg, Parallel Comput., 40:328–343, 2014] are presented and discussed. We show that the absolute error in the Frobenius norm behaves as $O\left(n^{1/2} \right), n \longrightarrow \infty $ and $O\left(\tau^{p/2} \right), \tau \longrightarrow 0,\,\, \forall p < 2$ for all three methods, where $n$ is the matrix size and $\tau$ is the truncation threshold.  We compare the  methods on a model problem and show that the combined method outperforms the original two. The methods are also applied to matrices coming from large chemical systems with $\sim 10^6$ atoms. We show that the combination of the two methods achieves better weak scaling by reducing the amount of communication by a factor of $\approx 2$.
\end{abstract}

\section{Introduction} \label{intro}

Although dense matrix-matrix multiplication is a well\Hyphdash studied procedure, its cubic complexity makes it very resource\Hyphdash demanding or even infeasible to use for large matrices. Much effort has been invested in finding better algorithms with reduced complexities, see for instance \cite{strassen1969gaussian}, \cite{coppersmith1987matrix}. When it comes to implementation of dense matrix multiplication, the way matrices are stored also affects the performance \cite{bader2006cache}.

For sparse matrices, there are many possible formats and it might be difficult to choose the one which provides the best performance. As a rule, in sparse matrix algorithms, matrices are stored in special formats, such as compressed sparse rows (CSR) or compressed sparse columns (CSC) \cite{saad2003iterative}. A drawback of these formats is that once one of them is chosen, operations with matrix transpose become difficult to perform in parallel. Using the idea of blocking, \citet{bulucc2009parallel} suggested the compressed sparse blocks (CSB) format, which solves this problem, at least partially. Another application of blocking technique is the blocked CSR format (BCSR), which works better for matrices with zeros occurring in blocks \cite{im2001optimizing}. All these storage formats and sparse matrix algorithms are designed for common-sense sparse matrices, i.e. such that have a few non-zero elements per row. 

The problem of our interest falls somewhere in between dense and sparse matrix multiplication. In electronic structure calculations based on the Hartree--Fock method or Kohn--Sham density functional theory, a key component is the computation of a density matrix $D$ for a given Hamiltonian $F$. A popular method is to perform a polynomial approximation of the function $D = \theta(\mu I - F),$ where $\theta$ is the Heaviside step function and $\mu$ is the chemical potential. This can be done in different ways. Two commonly used approaches are recursive application of low-order polynomials (density matrix purification, \cite{niklasson2002expansion, holas2001transforms}) and construction of Chebyshev polynomials \cite{goedecker1994efficient, baer1997sparsity, mohr2017efficient}. Depending on the implementation, the core operation is often matrix-matrix multiplication, and it is the crucial component for overall performance and scaling of the code. 

Matrices in electronic structure calculations have an important property of decay of elements. A matrix $A$ is said to obey an exponential decay with constants $c > 0,$ $\alpha > 0$ and $\lambda > 1$ with respect to $d(i,j)$ if $|a_{i,j}| \leq c \lambda^{-\alpha d(i,j)}$ or an algebraic decay with constants $c > 0$ and $\lambda > 1$ if $|a_{i,j}| \leq c \left(d(i,j)^\lambda + 1 \right)^{-1}$, where $d(i,j)$ is a distance function defined on the index set of the matrix. The distance function corresponds to the physical distance between some a priori chosen basis functions. It is convenient to look at such chemical system as a spatial graph, where the distance function is defined on the edges. Under certain assumptions, small matrix elements can be ignored, as usually done to matrices arising in linearly scaling electronic structure calculations \cite{li1993density, challacombe2000general, rubensson2011methods, bowler2010calculations}. In this context, linear scaling means that the computational time is directly proportional to the size of the considered chemical system. Such matrices are not sparse in the common sense, since they may have from several non-zero elements to several thousands non-zero elements per row, depending on the type of basis functions used, and thus standard sparse matrix algorithms may become inefficient here. However, the special structure of matrices with decay allows to achieve linear scaling, hence the name for the group of methods. Ideally, with linearly increasing computational power, it should give constant computational time. This is difficult to achieve due to communication costs, which start to play a significant role at some point.

\paragraph{Related work} \label{related_work}

The number of non-zero elements per row and decay properties dictate also that the compressed storage formats like CSR or CSC are unlikely to be usable in this context, because their strong point, i.e. compression, does not work well enough in this case. One alternative is introduced by \citet{mohr2017efficient}, where they use the so-called segment storage format (SSF). This format is based on the idea of grouping together consecutive non-zero entries in segments. Another approach is to exploit the hierarchical structure of a matrix, i.e. treat it as a matrix of matrices \cite{rubensson2007hierarchic}. Similar ideas are used in \cite{bock2013optimized, rubensson2016locality}. This representation is natural not only for the matrices which arise in electronic structure calculations, but in a wider class of problems, including, but not limited to, integral equations, partial differential equations, matrix equations and many more, see \cite{bebendorf2008hierarchical, hackbusch2015hierarchical} for extensive study. However, the notion of a hierarchical matrix used by \citet{hackbusch2015hierarchical} is different, since it is a data-sparse approximation of matrices, which are not necessarily sparse by construction.

Multiplication of matrices, which have more non-zero elements per row, than standard sparse matrices, or, in other words, nearly sparse matrices, is a challenging problem, especially since in general the sparsity pattern is not known in advance. In this case, load balancing becomes a substantial difficulty.

One of the first parallel implementations of nearly-sparse matrix multiplication is done by \citet{challacombe2000general}. The author uses conventional MPI and a special data format with distributed blocked compressed sparse rows. The load balancing and locality issues are addressed by introducing a space filling curve, which is a heuristic solution for the travelling salesman problem. The efficiency is demonstrate up to 128 cores.

An alternative approach is suggested by \citet{bulucc2012parallel}. They address load balancing issues by randomly permuting rows and columns, which leads to high efficiencies in the general case. This approach is successfully applied in electronic structure calculations \cite{borvstnik2014sparse,vandevondele2012linear}. In \cite{vandevondele2012linear} for instance, a scalability to $\sim 10^6$ atoms on 46656 cores is reported.

\citet{dawson2018massively} employ the communication-avoiding 3D multiplication algorithm by \citet{ballard2013communication} in their NTPoly code. The load balancing is addressed in the same manner using permutations as in \cite{bulucc2012parallel}. 

\citet{rubensson2016locality} present a way of multiplying sparse matrices in parallel, preserving an important property of locality. The approach is based on a hierarchical representation of matrices and a task-based parallel programming model called Chunks and Tasks \cite{CHT-PARCO-2014}. The load balancing issue is addressed by a scheduler, which utilizes a work stealing concept. A more detailed description of the model is presented in Section \ref{cht_model}.

Commonly, in electronic structure calculations, sparsity is maintained by truncation of small elements before and sometimes after multiplication. Different strategies of choosing those elements exist, see for example the discussions in \cite{goedecker1999linear, rubensson2009truncation}. Unlike the majority of methods, the algorithm considered in this article performs truncation of sub-matrix products to reduce the computational complexity.

\paragraph{Outline and structure}

In this article we consider different aspects of possible parallel implementations of multiplication of truncated matrices, sparse approximate matrix multiplication ($\textit{SpAMM}$) and their combination in a fully recursive task-based parallel environment called Chunks and Tasks \citep{CHT-PARCO-2014}. The main contributions of the article are the novel asymptotic absolute error analysis (Section \ref{error_control} and Appendix \ref{appendix_with_proofs}) and the combined technique, which has reduced communication, but preserves the accuracy (asymptotically) of the two original methods, and the implementation of all three methods.

The article is organized as follows. Section \ref{the_algorithm} contains a brief description of the sparse approximate matrix multiplication algorithm and previous implementations. Section \ref{error_control} contains derivations related to the error control. In Section \ref{inplementation_cht}, we describe our implementations within the Chunks and Tasks programming model, briefly discuss the programming model itself, and the leaf level library. Section \ref{results} contains a description of the experimental set-up, benchmarks and their results discussion. Some conclusions are given in Section \ref{conclusion}.

\section{The $\textit{SpAMM}$ algorithm} \label{the_algorithm}

The sparse approximate matrix multiplication algorithm introduced in \cite{challacombe2010fast, bock2013optimized} belongs to a family of $n$-body solvers, which use hierarchical approximations of sub-problems to reduce complexity, similarly to the famous fast multipole method by \citet{greengard1987fast}. The $\textit{SpAMM}$ algorithm exploits a hierarchical representation of the matrix and the decay property, which in this case is preserved for all levels in the hierarchy, to skip multiplication of small sub-matrices. The decay property is preserved from level to level because the distance function does not change. Assume that the matrices are represented as quad-trees such that

\begin{equation} \nonumber
A^t = \begin{pmatrix}
A^{t+1}_{0,0} & A^{t+1}_{0,1} \\ A^{t+1}_{1,0} & A^{t+1}_{1,1}
\end{pmatrix}, 
\end{equation} where $t$ is the level in the hierarchy, $t \in [1, \lceil\log_2(n)\rceil ],~ n$ is the matrix size. Level $\lceil \log_2(n) \rceil + 1$ contains single matrix elements. With 

\begin{equation} \nonumber
\| A^t \|_F = \sqrt{\sum\limits_{i=0}^{1}\sum\limits_{j=0}^{1}\| A^{t+1}_{i,j} \|^2_F},
\end{equation} the $\textit{SpAMM}$ algorithm can be written recursively as shown in Algorithm \ref{SPAMM_pseudocode}. Along with each matrix we keep its squared Frobenius norm and, since it is additive, it can be conveniently computed recursively starting from the lowest level. The parameter $\tau$ determines truncation of sub-matrix products, i.e. in the product space, whereas commonly truncation is applied on the input matrices, i.e. in the vector space. (see for example \cite{rubensson2011bringing}). 

\subsection{Serial implementations of $\textit{SpAMM}$}

Following the naive serial recursive implementation in \cite{challacombe2010fast}, a highly optimized serial version is implemented by the same authors in \citep{bock2013optimized}. It uses an optimized non-recursive kernel of $\textit{SpAMM},$ which can be introduced at any level of recursion, symbolic multiplication with linkless trees and SSE intrinsics. 

\begin{algorithm}[t]
\begin{algorithmic}[1]
\REQUIRE $A$, $B$, $\tau$
\ENSURE $C$
\IF{lowest level}
	\STATE \textbf{return} $C = A B$ \label{spamm_algline2}
\ENDIF
\FOR{i = 0 \TO 1 } 
	\FOR{j = 0 \TO 1 } 
		\IF{$\| A_{i,0} \|_F \| B_{0,j} \|_F \geq \tau $} \STATE { $T_0 =  SpAMM(A_{i,0}, B_{0,j},\tau)$ } \ELSE \STATE{$T_0 = 0$} \ENDIF
		\IF{$\| A_{i,1} \|_F \| B_{1,j} \|_F \geq \tau $} \STATE { $T_1 =  SpAMM(A_{i,1}, B_{1,j},\tau)$ } \ELSE \STATE{$T_1 = 0$} \ENDIF
		\STATE $C_{i,j} = T_0 + T_1$ 
	\ENDFOR	
\ENDFOR
\RETURN $C$
\end{algorithmic}
\caption{SpAMM}
\label{SPAMM_pseudocode}
\end{algorithm}

\subsection{Parallel implementations of $\textit{SpAMM}$}

\citet{bock2016solvers} present two parallel implementations of the $\textit{SpAMM}$ algorithm. The first one, which uses the OpenMP API, exploits parallel quad-tree traversal using untied \verb|task|,  i.e. a task which can be resumed by any thread (not necessarily by the one which started it) after suspension. The implementation demonstrates good parallel scalability with efficiency up to 80 \%. However, the authors notice that at some point load balancing becomes an issue with decreasing quad-tree height.

The second implementation by the same authors targets distributed memory machines and is done using the Charm++ parallel programming model \cite{kale1993charm++}. The implementation is tested on a large cluster with more than 24000 cores. A presence of modest serial component is noted, which is likely to be attributed to the Charm++ runtime. The load balancing problem is addressed by Charm++ built-in load balancer, which migrates chares (actors in the Charm++ terminology) between the nodes taking into account communication costs.

An approach similar to $\textit{SpAMM}$ and referred to as  on-the-fly filtering of the product matrix is also used in \cite{borvstnik2014sparse}. The authors note that it can provide a speedup of 300\% on realistic matrices coming from discretization of chemical systems with up to $10^6$ atoms on around 5100 cores. They also combine it with removing small blocks from the product matrix. However, no error analysis is given.

\section{Error estimate} \label{error_control} As noted earlier, the $\textit{SpAMM}$ algorithm performs truncation of sub-matrix products, i.e. operates in the product space. Truncation inevitably brings some error in the product matrix. The authors of the original algorithm note that the dependency between the  parameter $\tau$ in the algorithm and the resulting error in the product matrix is not known, see \citep[p. C84]{bock2013optimized}. With the modified version of $\textit{SpAMM}$ \citep{2015arXiv150805856C}, it is possible to derive the relative error bound and to show stability of the process in the sense of \cite{demmel2007fast}. However, this bound does not make any assumptions on the matrix and it is hardly convenient for practical use, since the bound depends on $n^2$, where $n$ is the matrix size.

In Section \ref{intro}, we mentioned the decay property of matrices, which are common in electronic structure calculations. We exploit this property, and ignore matrix elements smaller than a certain threshold value to derive error estimates. In this section we show the asymptotic behavior of the absolute error in the product matrix. We limit ourselves to the case of matrices with exponential decay and state novel lemmas and theorems. The proofs can be found in Appendix \ref{appendix_with_proofs}.

Define first a distance function. We use the same framework as in \cite{Rubensson2018localized}. The function $d(i,j)$ is a distance function on the index set $I = \{1,\ldots,n\}$ if for any $i,j,k \in I$
\begin{enumerate}
\itemsep0em
\item $d(i,j) \geq 0;$
\item $d(i,j) = d(j,i);$
\item $d(i,i) = 0;$
\item $d(i,j) \leq d(i,k) + d(k,j).$
\end{enumerate} In other words, $d(i,j)$ is a pseudo-metric defined on $I.$

\citet{Rubensson2018localized} consider a sequence of matrices $\{A_n\}$ with exponential decay with the same constants $c > 0$ and $\alpha > 0$ and associated distance functions $d_n(i,j)$. Under the assumption $|N_{d_n}(i,R)| \leq \gamma R^{\beta},\,\ \forall R > 0,\, \forall i,$ where $N_{d_n}(i,R)$ is the set of all indices within a distance $R$ away from the $i$-th index and $|N_{d_n}(i,R)|$ is its cardinality, $\gamma > 0$ and $\beta > 0$ are some constants independent of $n$, they show that for any $\varepsilon > 0,$ and for any $n > 0,$ the matrix $A_n$ contains at most $O(n)$ elements greater than $\varepsilon$ in magnitude. Moreover, each row and column of each $A_n$ has a number of entries greater than $\varepsilon$ in magnitude bounded by a constant independent of $n$, see Theorem 4 in \cite{Rubensson2018localized}. The assumption means that the number of indices situated within a distance $R$ away from an index is finite, which holds for the underlying physical systems. The theorem states that each row and column of such a matrix has at most $\kappa$ significant elements, i.e. elements with magnitude greater than $\varepsilon$, where $\kappa$ is a constant independent of $n$. The rest of the elements are not significant, but not necessarily zeros. We directly exploit these results to derive our estimate. 

\begin{definition} \label{def:decay} A sequence of $n \times n$ matrices $\{ A_n \}$ is told to have \textit{exponential decay w.r.t. associated distance functions} $\{d_n(i,j)\}$ for each $n$ if there exist positive constants $\alpha, c$ independent of $n$ such that

\begin{equation} \label{eq:exp_decay_single_sequence}
| [A_n]_{i,j} | \leq c e^{-\alpha d_n(i,j)}\,\,\, \forall i,j=1,\ldots,n.
\end{equation} 

\end{definition}

\begin{definition} \label{def:decay_common} Two sequences of $n \times n$ matrices are told to have \textit{exponential decay w.r.t. common associated distance functions} ${d_n(i,j)}$ for each $n$ if there exist positive constants $\alpha, c_1, c_2$ such that

\begin{equation} \nonumber
| [A_n]_{i,j} | \leq c_1 e^{-\alpha d_n(i,j)},\,\,| [B_n]_{i,j} | \leq c_2 e^{-\alpha d_n(i,j)} \,\,\, \forall i,j=1,\ldots,n.
\end{equation}

\end{definition}

\begin{definition} \label{def_finite_num_vertices} A sequence of $n \times n$ matrices $\{ A_n \}$ with exponential decay w.r.t. associated distance functions $\{d_n(i,j)\}$ is told to have a finite number of indices around each index for any $n$ if there exist positive constants $\gamma$ and $\beta$ independent of $n$ such that 

\begin{equation} \nonumber 
| N_{d_n}(i,R) | \leq \gamma R^{\beta},\,\,\forall i = 1,\ldots,n,\,\,\forall R > 0,
\end{equation} where $N_{d_n}(i,R)$ is the set of all indices within a distance $R$ away from the $i$-th index and $|N_{d_n}(i,R)|$ is its cardinality.

\end{definition}

\begin{lemma} \label{lemma_sum_insignificant}
Let $\{ C_n \}$ be  a sequence of $n \times n$ matrices satisfying definition \ref{def:decay} and assume it also satisfies definition \ref{def_finite_num_vertices}. Then, for any $\varepsilon > 0$ and every $C_n$,

\begin{equation} \nonumber 
\sum\limits_{i,j:\, | [C_n]_{i,j} | \leq \varepsilon} | [C_n]_{i,j} |^2 = \begin{cases} O(n), & n \longrightarrow \infty, \\ O(\varepsilon^p),\, \forall p < 2, & \varepsilon \longrightarrow 0. \end{cases}
\end{equation}

\end{lemma}

The result of Lemma \ref{lemma_sum_insignificant} should be interpreted as follows: the sum of squares of all elements smaller than $\varepsilon$ in magnitude grows at most proportionally to $n$ when $\varepsilon$ is fixed and $n$ grows, and goes to zero as fast as $\varepsilon^p,\, \forall p < 2$ or faster when $n$ is fixed.  

First we consider the error introduced by truncation of the input matrices, i.e. truncation in the vector space.

\begin{theorem} \label{theorem_truncmul_new}

Let  $\{ A_n \}$ and  $\{ B_n \}$ be sequences of $n \times n$ matrices satisfying definition \ref{def:decay_common}. Assume both sequences satisfy also definition \ref{def_finite_num_vertices}. Let $C_n = A_n B_n$ and $\tilde{C}_n = \tilde{A}_n \tilde{B}_n,$ where
\begin{equation} \label{truncated_matrices}
\begin{split}
[\tilde{A}_n]_{i,j} & = \begin{cases} 0, &\text{if } |[A_n]_{i,j}|  < \tau, \\ [A_n]_{i,j} &\text{otherwise}, \end{cases} \\
[\tilde{B}_n]_{i,j} & = \begin{cases} 0, &\text{if } |[B_n]_{i,j}|  < \tau, \\ [B_n]_{i,j} &\text{otherwise,}  \end{cases}
\end{split}
\end{equation} are the truncated matrices, $\tau > 0$ is some given parameter. Then, the error matrix $E_n = C_n - \tilde{C}_n$ has the following element-wise property:
\begin{equation} \nonumber
| [E_n]_{i,j} | = O(\tau),\,\,\forall i,j = 1,\ldots,n. 
\end{equation} Moreover, 

\begin{equation} \nonumber
\| E_n \|_F = \begin{cases} O(n^{1/2}), & n \longrightarrow \infty, \\ O(\tau^{p/2}),\, \forall p < 2, & \tau \longrightarrow 0. \end{cases}
\end{equation}

\end{theorem}

The result of Theorem \ref{theorem_truncmul_new} should be read as follows: the absolute value of each element in the error matrix decays linearly or faster with $\tau$ and the Frobenius norm of the error matrix grows at most as $n^{1/2}$ for fixed $\tau$ and decays as $\tau^{p/2},\,\forall p < 2$ or faster when $n$ is fixed.

The $\textit{SpAMM}$ algorithm has the same behavior of the error, as in the multiplication of truncated matrices.

\begin{theorem} \label{theorem_spamm_new}

Let  $\{ A_n \}$ and  $\{ B_n \}$ be sequences of $n \times n$ matrices satisfying definition \ref{def:decay_common}. Assume both sequences satisfy also definition \ref{def_finite_num_vertices}. Let $C_n = A_n B_n$ and $\bar{C}_n = SpAMM(A_n,B_n,\tau),$ where $\tau > 0$ is some given parameter. Then, the error matrix $E_n = C_n - \bar{C}_n$ has the following element-wise property:
\begin{equation} \nonumber
| [E_n]_{i,j} | = O(\tau),\,\,\forall i,j = 1,\ldots,n. 
\end{equation} Moreover, 

\begin{equation} \nonumber
\| E_n \|_F = \begin{cases} O(n^{1/2}), & n \longrightarrow \infty, \\ O(\tau^{p/2}),\, \forall p < 2, & \tau \longrightarrow 0. \end{cases}
\end{equation}

\end{theorem}

If one performs truncation in both vector and product spaces, i.e. applies the $\textit{SpAMM}$ algorithm on truncated matrices, then the norm of the error matrix still behaves similarly.

\begin{theorem} \label{theorem_tuncated_spamm_new}
Let  $\{ A_n \}$ and  $\{ B_n \}$ be sequences of $n \times n$ matrices satisfying definition \ref{def:decay_common}. Assume both sequences satisfy also definition \ref{def_finite_num_vertices}. Let $C_n = A_n B_n$ and $\hat{C}_n = SpAMM(\tilde{A}_n,\tilde{B}_n,\tau),$ where $\tilde{A}_n$ and $\tilde{B}_n$
are the truncated matrices \eqref{truncated_matrices} and $\tau > 0$ is some given parameter. Then, the error matrix $E_n = C_n - \hat{C}_n$ has the following element-wise property:
\begin{equation} \nonumber
| [E_n]_{i,j} | = O(\tau),\,\,\forall i,j = 1,\ldots,n. 
\end{equation} Moreover, 

\begin{equation} \nonumber
\| E_n \|_F = \begin{cases} O(n^{1/2}), & n \longrightarrow \infty, \\ O(\tau^{p/2}),\, \forall p < 2, & \tau \longrightarrow 0. \end{cases}
\end{equation}
\end{theorem}

We show in Section \ref{results} that in a distributed environment, it is profitable to combine multiplication of truncated matrices and $\textit{SpAMM}$ into a new technique, referred to as hybrid.

\section{Implementation details} \label{inplementation_cht} We implement the three algorithms using the Chunks and Tasks programming model \cite{CHT-PARCO-2014}. In this section we briefly introduce the model and its concepts, describe the matrix representation, and discuss implementation issues. 

\subsection{The Chunks and Tasks programming model} \label{cht_model}

The Chunks and Tasks model belongs to a family of task-based programming models. It is defined by a C++ API. The basis of the model comprises two basic abstractions - a chunk and a task. A chunk represents a piece of data whereas a task represents a piece of work. The parallelism is exploited in both work and data dimensions.

The data is incapsulated in chunks objects, which are registered to the runtime library. The user operates chunk identifiers, which are obtained after chunk registration. Once a chunk is registered, its modification is no longer possible (similarly to the Concurrent Collections model \cite{budimlic2010concurrent}). Chunk identifiers can be provided as input to tasks or be a part of other chunks. This allows to build hierarchies of chunks and the runtime takes responsibility of managing such hierarchies.

The work is expressed in terms of tasks. A task has one or more input chunk identifiers and a single output chunk identifier. Access to chunk data is possible only as input to a task, and one cannot execute a task without having all input chunks available. The user is free to register other tasks inside a task, giving rise to dynamic parallelism.

A proof-of-concept runtime library \citep{CHT-PARCO-2014} is implemented in C++ and internally parallelized with MPI and Pthreads. A Chunks and Tasks program is executed by a group of worker processes. To address load balancing, the implementation uses the concept of work-stealing, similarly to XKaapi \cite{gautier2013xkaapi}, Cilk \cite{blumofe1996cilk}, StarPU \cite{augonnet2011starpu}.

The set of rules introduced by the model is intended to simplify parallel programming for distributed memory machines by reducing the chance of getting a deadlock to zero. Moreover, immutable data prohibits a race condition of any kind. However, it takes a certain effort for a conventional MPI programmer to shift the programming paradigm. 

\subsection{Matrix representation} \label{matrix_representation}

As we discussed in the beginning of the section, in the Chunks and Tasks model all data is encapsulated in chunk objects. Matrices are no exception. Similarly to \citep{rubensson2007hierarchic}, we employ a hierarchical matrix representation. Not only is it naturally suitable for recursive algorithms like $\textit{SpAMM}$ or localized inverse factorization \cite{Artemov2018parallelization}, but it is easily mapped onto the Chunks and Tasks model. \citet{rubensson2016locality} present a matrix library implemented using the model (CHTML). A matrix is viewed as a quadtree, and all non-leaf nodes are chunks which contain identifiers of child nodes. Matrices containing no non-zero elements are not kept in the memory, they are represented by a special identifier value. Matrix elements are kept at the leaf level only. Actual computations take place only if the leaf level is reached. This approach allows to skip computations on zero sub-matrices, and thus exploit sparsity patterns dynamically. 

The leaf level matrices are in shared memory and can be organized in different ways. One could keep them dense, or use a block-sparse representation, where non-zero blocks are kept in a 2-D array, see \cite{rubensson2016locality} for details and discussion. For efficient computations, leaf level routines rely on an efficient implementation of BLAS (Basic Linear Algebra Subroutines \cite{dongarra1990set}), which is usually provided by a vendor or on another high-performance library which handles basic operations.

In order to be consistent with the $\textit{SpAMM}$ algorithm, each matrix stores its own squared Frobenius norm. This norm is to be computed when the matrix is constructed as an output of a task.

\subsection{SpAMM task type}

We express the $\textit{SpAMM}$ algorithm in terms of the Chunks and Tasks matrix library. The main difference from matrix-matrix multiplication introduced by \citet{rubensson2016locality} is that the task checks Frobenius norms of multipliers and if their norm product is small enough, the output chunk is set to null. Otherwise the same task is registered for sub-matrices and the process continues recursively until reaching the leaf level. The $\tau$ parameter is provided as one of the input chunks.

\subsection{Leaf level library} 

As we discussed in Section \ref{matrix_representation}, CHTML can be used together with different leaf level representations. The previous works \citep{rubensson2016locality, Artemov2018parallelization} use a block-sparse leaf level library. However, it does not suite the $\textit{SpAMM}$ algorithm, since it is not consistent with the hierarchical structure used in the algorithm. 

In order to be fully consistent, we develop a new leaf level library for this work, namely the hierarchical block-sparse library, which relies on a hierarchical matrix representation as the name suggests. The library represents a matrix as a quad-tree and zero matrices as null pointers. The hierarchy does not go down to single matrix elements. It stops when reaching a certain predefined block size or if a submatrix is identically zero. A dense lowest-level block is kept if it contains at least one non-zero value. Hence, in some sense it exploits blocking ideas similar to block-sparse leaf matrix in \cite{rubensson2016locality,bulucc2012parallel}. 

The advantage of the hierarchical block-sparse representation over the previously used flat block-sparse library is that with the latter one always has to iterate through a large number of small blocks and look for a non-zero block before accessing it. In the hierarchical representation, larger zero blocks are skipped higher in the hierarchy when accessing a non-zero block. 

In real applications most of the multiplication rejections in the $\textit{SpAMM}$ algorithm do happen for small sub-matrices. That is partially due to the nature of the Frobenius norm and, in principle, could be addressed by making the leaf level size smaller and using another leaf library (for instance, the aforementioned block-sparse one, which is not hierarchical), but this approach leads to great administrative overhead, as there will be more levels in the chunk hierarchy on the distributed level, and, consequently, performance degradation. Our solution is to make the algorithm also operate at the leaf level and the new library addresses that.

Currently, the hierarchical block-sparse matrix library supports a set of matrix operations including multiplication, addition, rescaling, computation of inverse factor and $\textit{SpAMM}$. Corresponding BLAS operations are called at the lowest level of the hierarchy. As one can see, the set of operations is consistent with the functionality of the Chunks and Tasks matrix library (CHTML), which can be found in \cite{rubensson2016locality}. This is done for compatibility purposes. 

There are two ways how to do multiplication: the first is to multiply \textit{when} building the output matrix, i.e. to allocate the hierarchical structure on-the-fly, the second is to build the whole matrix quad-tree \textit{in advance}, and then to  perform all the lowest-level multiplications in a batched way and populate the leafs of the output matrix with results. In experiments we always use the latter approach. It is implemented using the \verb|unordered multimap| container from the C++ STL library. Currently we assume that the leaf library is entirely serial, but, if needed, the batched multiplication can be efficiently parallelized using the C++ Parallelism TS, which is a part of the C++17 standard, but not yet supported in all compilers. 

Since the hierarchical block-sparse leaf matrix library operates in shared memory, some routines can be efficiently improved compared to the distributed execution with the Chunks and Tasks.

Consider now how the multiplication is done in the Chunks and Tasks matrix library, described in Algorithm \ref{MM_pseudocode}. At line \ref{matmul_algline1} the input is checked and if any of the input matrices is zero (represented as a special null identifier), then the output matrix identifier is also set to null. Lines \ref{matmul_algline2}--\ref{matmul_algline4} describe leaf-level operations and lines \ref{matmul_algline5}--\ref{matmul_algline14} describe higher level operations. The code is organized in this manner due to one of the imposed limitations: a chunk content can be accessed only as an input to a task, i.e. only a single level of the hierarchy is visible at a time. In other words, if neither of the two input matrices are zero, then 13 tasks (8 multiplications, 4 additions, creating a matrix out of child identifiers) will be anyway registered, unless the lowest level is involved. Registration of tasks definitely leads to overhead.

\begin{algorithm}[t]
\begin{algorithmic}[1]
\REQUIRE $A, B$
\ENSURE $C$
\IF{$A$ not NULL \textbf{and} $B$ not NULL} \label{matmul_algline1}
\IF{lowest level}	\label{matmul_algline2}
	\STATE{ $X = $ leafMatrixMultiply$(A,B);$} \label{matmul_algline3}
	\STATE{ $C$ = registerChunk$(X)$; } \label{matmul_algline4}
\ELSE \label{matmul_algline5}
	\FOR{$m=1$ \TO 2 } \label{matmul_algline6}  
		\FOR{$n=1$ \TO 2 } \label{matmul_algline7}
			\STATE{$Y_1 = $ registerTask(MatrixMultiply, $A_{m,1}, B_{1,n}$);} \label{matmul_algline8}
			\STATE{$Y_2 = $ registerTask(MatrixMultiply, $A_{m,2}, B_{2,n}$);} \label{matmul_algline9}
			\STATE{$C_{m,n} = $ registerTask(Add, $Y_1, Y_2$);} \label{matmul_algline10}
		\ENDFOR \label{matmul_algline11}	 
	\ENDFOR \label{matmul_algline12}
	\STATE{$C = \begin{bmatrix} \label{matmul_algline13}
	C_{1,1} & C_{1,2} \\ C_{2,1} & C_{2,2}
	\end{bmatrix};$}
\ENDIF \label{matmul_algline14}
\ELSE \label{matmul_algline15}
\STATE{$C = $ NULL;} \label{matmul_algline16}
\ENDIF \label{matmul_algline17}
\STATE{\textbf{return} $C$} \label{matmul_algline18}
\end{algorithmic}
\caption{MatrixMultiply}
\label{MM_pseudocode}
\end{algorithm}

In a shared memory environment the multiplication of hierarchical matrices can be improved. Consider an illustrative example of such multiplication in \eqref{nonoptimal_MM}. As one can see, although none of the input matrices is empty, the resulting matrix is empty. So, it is not worth to multiply those matrices at all. However, we could only detect this when seeing the whole hierarchy, but not the two matrix identifiers.

Since the leaf level library works in shared memory, it is indeed possible to see the whole hierarchy at once and predict if it is worth to start multiplication of two matrices or to do a memory allocation in case of batched multiplication. This check should be done recursively, since it might happen that the resulting matrix $C$ does have a non-zero entry, see (\ref{nonoptimal_MM2}). In this case, prediction also should be done when multiplying $A_1$ by $B_3$. If those matrices give non-zero product, then the multiplication routine is invoked.

In our implementation, such prediction is utilized not only for the multiplication routine, but also for the others, including the $\textit{SpAMM}$ routine, where the conclusion is based not only on the layout of children and recursive checks, but also on the norms of corresponding matrices. This allows our library to be competitive in terms of performance with the block-sparse library used in the previous works. More details can be found in Section \ref{results}.

Though predictor routines are called before every call to multiplication and thus extra traversals of the quad-tree are done, the overhead is small compared to the overhead of memory allocation for new matrix objects, which anyway would not be used. Prediction is cheap because it involves only pointer arithmetic and logical operations. For instance, we generate two matrices of size 2048 with non-zero random blocks distributed randomly uniformly with block-size 64. With fill-in factor (i.e. the proportion of non-zero elements in the matrix) 0.2 multiplication takes approximately 0.036 seconds, and the corresponding prediction takes $2.71\times 10^{-5}$ seconds. For fill-in factor 0.7, multiplication takes approximately 0.44 seconds, whereas prediction takes $4.84 \times 10^{-5}$ seconds, i.e. several orders of magnitude less.

\begin{equation} \label{nonoptimal_MM}
\begin{bmatrix}
A_0 & \textit{NULL} \\ \textit{NULL} & \textit{NULL}
\end{bmatrix} \times
\begin{bmatrix}
\textit{NULL} & \textit{NULL} \\ \textit{NULL} & B_3
\end{bmatrix} = 
\begin{bmatrix}
\textit{NULL} & \textit{NULL} \\ \textit{NULL} & \textit{NULL}
\end{bmatrix}
\end{equation}

\begin{equation} \label{nonoptimal_MM2}
\begin{bmatrix}
A_0 & A_1 \\ \textit{NULL} & \textit{NULL}
\end{bmatrix} \times
\begin{bmatrix}
\textit{NULL} & \textit{NULL} \\ \textit{NULL} & B_3
\end{bmatrix} = 
\begin{bmatrix}
\textit{NULL} & A_1 \times B_3 \\ \textit{NULL} & \textit{NULL}
\end{bmatrix}
\end{equation}

\subsection{The full two-level algorithm} The full implementation of $\textit{SpAMM}$ then reads as follows: at the level of the Chunks and Tasks library, the recursive task-based version of Algorithm \ref{SPAMM_pseudocode} is utilized. The matrices are split into hierarchies until a predefined relatively large size (so that one task could occupy a whole computational node) is reached. We further refer to the corresponding size as to \textit{task size}. Inside those blocks the matrices are kept in hierarchies as well using the hierarchical block-sparse leaf level library. When the task-based code eventually arrives to the base case after several recursive calls, the corresponding $\textit{SpAMM}$ method of the leaf library is called. Actual computations are performed there and the resulting block products are computed. Then, a chunk hierarchy of the resulting product is built up starting from leafs. 

\section{Results} \label{results} 

We provide a set of experiments, carried out in order to evaluate the performance of the leaf level library and the approximate multiplications algorithms. Here and further we refer to multiplication of truncated matrices as "Truncmul," $\textit{SpAMM}$ as "SpAMM" and their combination as "Hybrid."

\subsection{Experimental setup} \label{experiments}

All experiments are done at PDC high performance computing center located at KTH Royal Institute of Technology in Stockholm. The target machine is the Beskow cluster, which is a Cray machine equipped with 2060 nodes each containing 2 16-core Intel Xeon E5-2698v3 CPUs combined with 64 gigabytes of RAM. The base CPU frequency (without boost) is 2.3 GHz. The nodes are connected in Dragonfly topology using the Cray Aries high speed network. The code is compiled with GCC 8.3.0 g++ compiler, Cray MPICH 7.7.8 and OpenBLAS 0.3.9 \cite{OpenBLAS}. The OpenBLAS library is built from the source code with disabled multi-threading. This is highly recommended by the developers in case an application uses multi-threading somewhere else, which is the case with the Chunks and Tasks library. 

\subsection{Performance of the leaf level library} In this benchmark the aim is to investigate the performance of the new hierarchical block-sparse matrix library and compare it with the existing block-sparse library. For this, we perform matrix-matrix multiplication at the leaf level, and the Chunks and Tasks library is not used at all. Two matrices of size $2048 \times 2048$ of varying fill-in factor are multiplied. The non-zero blocks of the matrices are distributed randomly uniformly over the matrix. All computations are performed in double precision. The total number of multiplications is computed by counting calls to BLAS \verb|gemm| routine. As a reference when reporting the performance of our implementation, we use a theoretical peak performance model based on \cite{dolbeau2018theoretical}. It takes into account a variety of instruction sets, also known as intrinsics. Modern compilers automatically generate codes which use intrinsics. The target CPU belongs to the Intel Haswell family, thus it has four different instruction sets: SSE, AVX, AVX2 and FMA. The last two are often referred as the same, but technically they are distinct, since FMA instructions have their own flag. According to the performance model, the peak performance of a single core can be computed as follows (case AVX+FMA (DP) of Table 4 in \cite{dolbeau2018theoretical}):
\begin{equation}
P = F \times \frac{flops}{operation} \times \frac{operations}{instruction} \times \frac{instructions}{cycle}, \label{peak_perforance_single_core}
\end{equation} where $F$ is the CPUs frequency in Hz, which depends on the number of active cores and the instruction set ($3.3 \times 10^9$ Hz for a single core executing AVX and $2.5 \times 10^9$ Hz for all cores executing AVX \cite{IntelXeonSpecs}), $flops~/$ $operation = 2$ is the number of floating point operations included in the instruction's semantic (commonly, fused operation like multiply-add are used, hence the value), $operations~/$ $ instruction = 4$ is the number of operands in the SIMD operation, which is $256 / 64 = 4$ for AVX registers operating with doubles, $instructions~/$ $ cycle = 2$ is the number of simultaneously executed instructions in a single cycle. The latter quantity is rather complicated to compute, but, as mentioned in \cite{dolbeau2018theoretical}, it can be approximated well enough by the presented value. The formula \eqref{peak_perforance_single_core} gives the peak performance of a single CPU core equal to $3.3 \text{GHz}\, \times 2 \times 4 \times 2\, \text{flops/cycle}  = 52.8$ Gflop/s, thus the overall peak performance of the CPU is $2.5 \text{GHz}\, \times\, 2\, \times\, 4\, \times\, 2\, \text{flops/cycle}\, \times\, 16\, \text{cores} = 640$ Gflop/s, which is consistent with CPU benchmark results available on the Internet \cite{IntelXeonSpecs}. We also measure the practical peak performance of the OpenBLAS \citep{OpenBLAS} library on a single core, which is $\approx 42$ Gflop/s, by multiplying two random dense matrices of size 2048 several times. Taking into account the frequency drop from $3.3$ GHz to $2.5$ GHz when all cores execute AVX instructions, the practical peak performance of the CPU is $\frac{42}{52.8} \times 100 \% \approx$ 80\% of the theoretical peak, i.e. $512$ Gflop/s.
 
\begin{figure}[h]
\center{\includegraphics[width=0.9\linewidth]{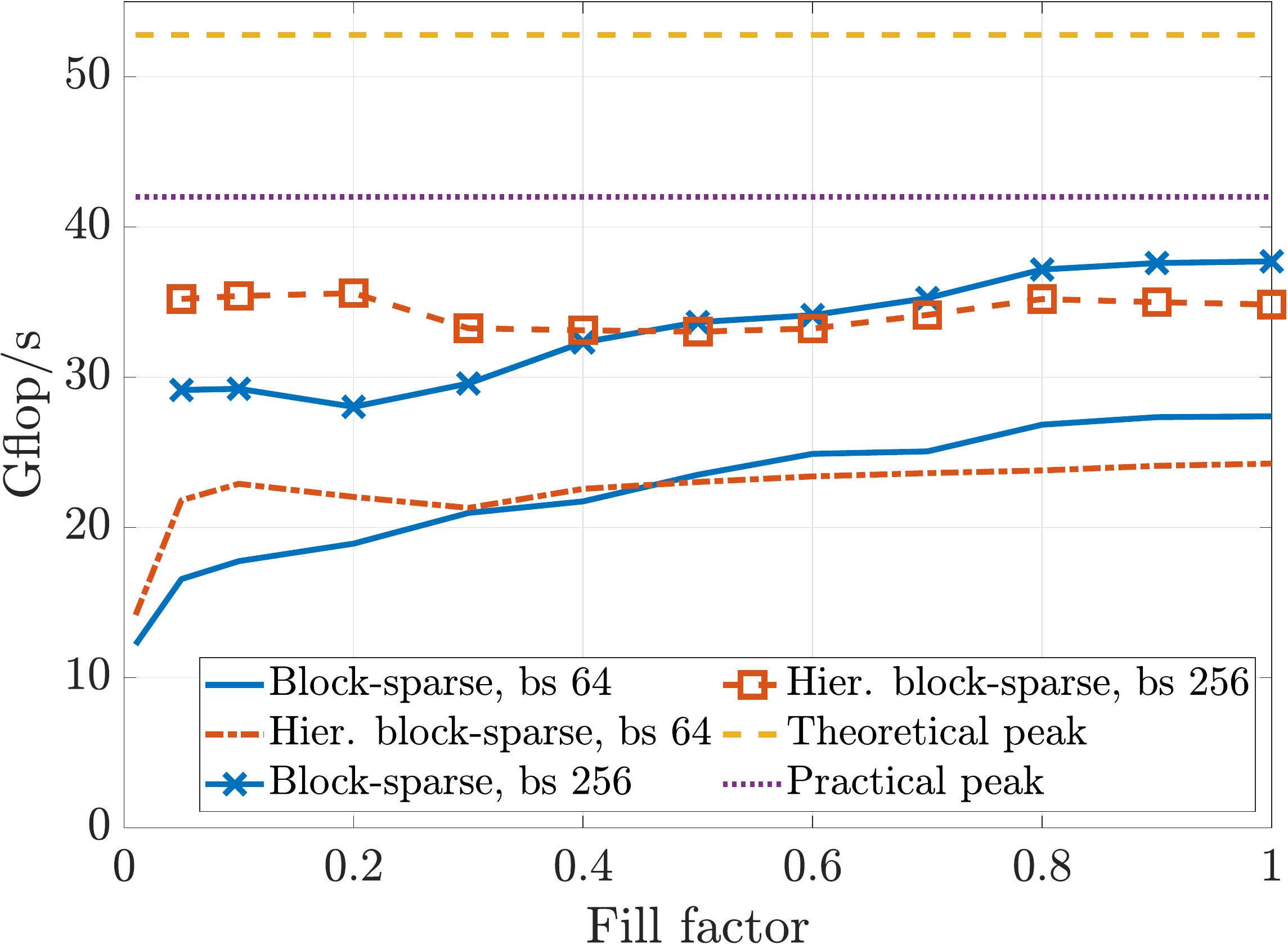}}
\caption{Leaf level libraries performance on a single core measured on a matrix-matrix multiplication of size $2048 \times 2048$ with double precision, two different blocksizes used. The non-zero random blocks were randomly uniformly distributed over the matrices.}
\label{ris:bsm_vs_hbsm}
\end{figure}

The benchmark results for two particular blocksizes are presented in Figure \ref{ris:bsm_vs_hbsm}. As one can observe, the hierarchical block-sparse library delivers almost a constant performance with increasing fill-in factor, whereas the block-sparse library has increasing performance. Clearly, there is a cross-point for both presented block-sizes, which indicates that the hierarchical library is superior for matrices with fill-in factors smaller than 0.3. The theoretical and practical peak performances are plotted in dashed lines for reference. For $bs = 256$ the hierarchical block-sparse library attains almost 65\% of theoretical peak performance for any fill-in factor, while the block-sparse library attains up to 70\% depending on the fill-in factor. Note that in real applications, the leaf library handles blocks of a larger distributed matrix, and therefore the fill-in factor varies from block to block.

\subsection{Performance on a model problem}

The aim of this benchmark is to investigate the performance of the Chunks and Tasks implementations of the three methods considered in Section \ref{error_control}. 

We create a matrix with exponential decay away from the main diagonal, i.e. $|A_{i,j}| \leq c e^{-\alpha |i-j|}$ with $c = 1$ and the decay rate $\alpha = 0.005$. Elements smaller than $10^{-16}$ in magnitude are set to zero by construction. The matrix is then multiplied by itself with the three methods described above. With Truncmul, no multiplication operation is skipped when computing the result, and all calls to \verb|gemm| are accounted. We sum calls to BLAS \verb|gemm| routine and then calculate the the total number of floating point operations as 

\begin{equation}
flops = 2 \times bs^3 \times N_{gemm},
\end{equation} where $bs$ is the size of the block at the lowest level of the leaf-level library, i.e. blocks of size $bs$ are given to BLAS routines, $N_{gemm}$ is the total number of calls to the \verb|gemm| routine. Once the total number of floating point operations is known, one can calculate the performance of the Truncmul method.

In turn, $\textit{SpAMM}$ and Hybrid methods skip multiplications of sub-matrices of various sizes depending on the parameter $\tau$ and it is not quite correct to compare their performances in Gflop/s. Instead, we propose the following: we select a target accuracy $\sigma$ of the final product matrix, i.e. the maximum Frobenius norm of the difference between the approximate and true products. For each of the algorithms, we perform a set of tests with varying threshold $\tau$ and pick the largest one such that the error matrix norm does not exceed $\sigma.$ We measure the time in each case and compare the timings. To have a reference to the theoretical/practical peaks, we provide the performance results in Gflop/s for Truncmul.

We assume that the Frobenius norm on each level of the matrix hierarchy is given (precomputed). Therefore we do not include the time needed to compute it to the total time. In practice, getting this axillary takes time, but in the hierarchical approach it can be done rather efficiently.

\paragraph{Single node performance} For this benchmark, we employ the CHT-MPI Chunks and Tasks library implementation and the CHTML matrix library. We set a single worker process to occupy the whole computational node and to perform all the tasks and no communication is involved. Within the worker, the tasks are executed in parallel, if possible, by 31 available threads, whereas a single thread is left for communication just to be consistent with the multiple node case. We set $n = 4 \times 10^4,$  the task size, i.e. the size where a piece of matrix processed by the leaf level library, is 2048, $bs = 32$, 64, 128, 256. The parameters are chosen so that the computation fits the amount of RAM available in a single node. We apply nine different values of $\tau$, namely $10^{-4},\,10^{-5}, \ldots, 10^{-12}.$  The target accuracy is $\sigma = 10^{-6}.$ The average over two turns of this benchmark test is presented in Figure \ref{ris:model_problem_perfromance}, left panel. As one can observe, Truncmul reaches around 33\% of the practical performance peak. If timings are considered, then it can be seen that both $\textit{SpAMM}$ and Hybrid are around 25--30 \% faster than Truncmul. Once can also see in Figure \ref{ris:model_problem_perfromance}, right panel, how different the  number of calls to \verb|gemm| is for Truncmul compared with the other two methods: the two latter methods generate about 40 \% fewer calls and \textit{SpaMM} generates slightly fewer calls than the Hybrid method. It also turns out that the threshold $\tau$ should be $10^{-10}$ for block-sizes $32, 64, 128$ and decreased to $10^{-11}$ for block-size 256 to preserve the accuracy.

\begin{figure*}[ht]
\begin{minipage}[h]{0.45\linewidth}
\center{\includegraphics[width=1\linewidth]{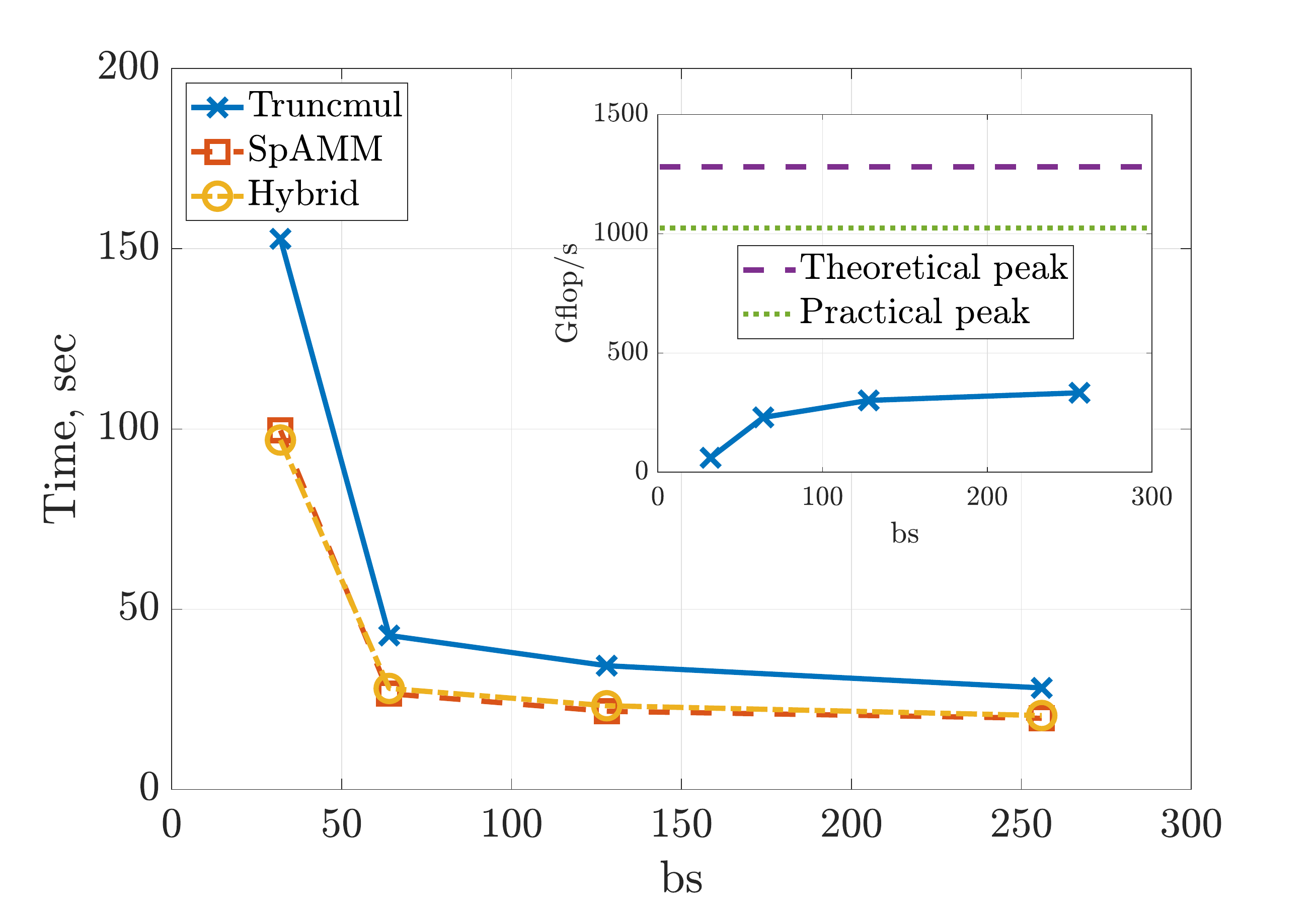}} 
\end{minipage}
\hfill
\begin{minipage}[h]{0.45\linewidth}
\center{\includegraphics[width=1\linewidth]{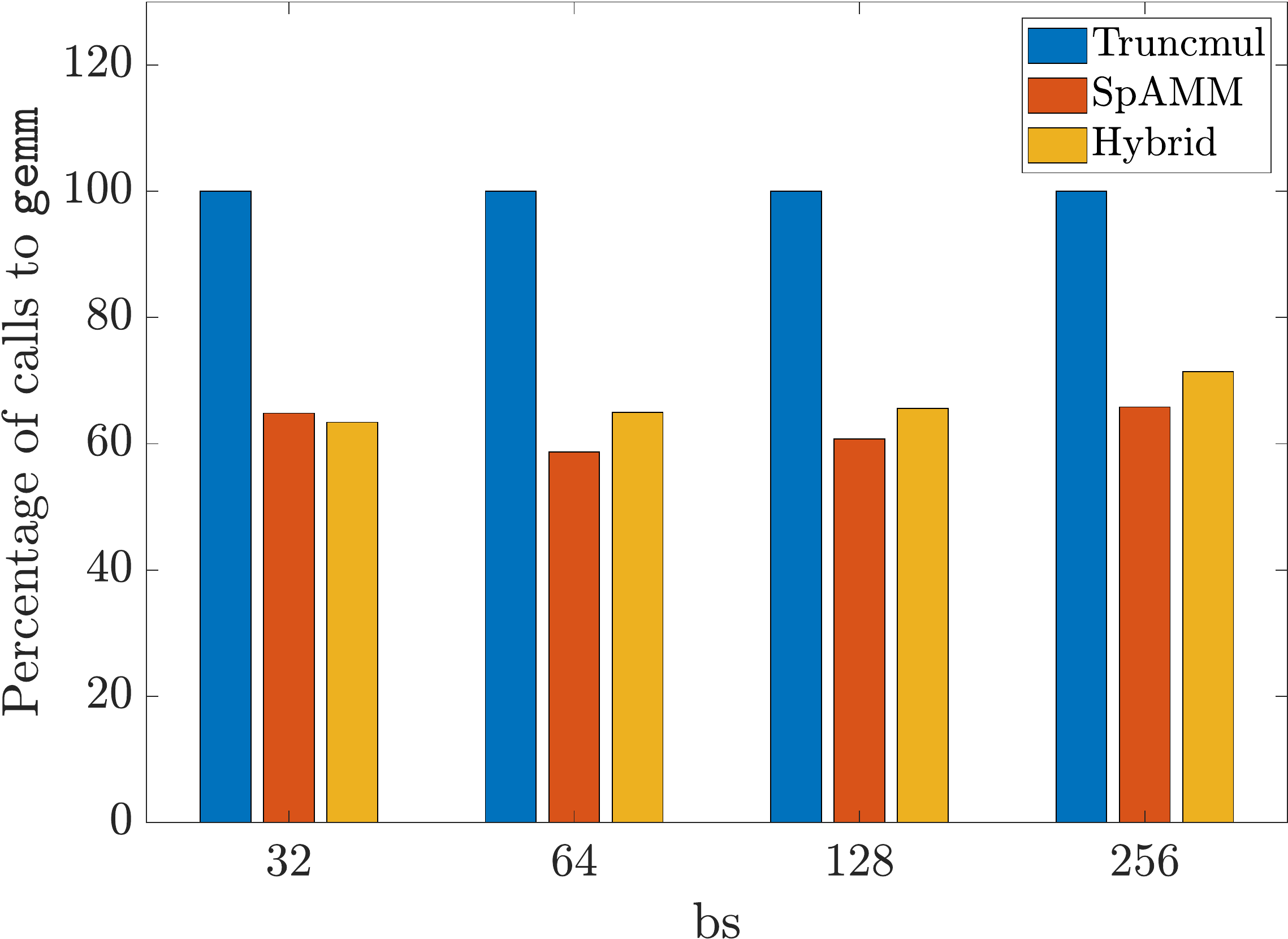}} 
\end{minipage}
\caption{Left panel: single node timings for the three methods (main plot), performance of Truncmul (subplot), target accuracy $\sigma = 10^{-6}$. Right panel: percentage of the number of calls to gemm BLAS matrix multiply routine compared to Truncmul. A matrix of size $n = 4\times 10^4$ is multiplied by itself. Block-sizes $bs = 32, 64, 128, 256$ are used and truncation is performed with $\tau = 10^{-10}$ for $bs = 32,64,128$ and $\tau = 10^{-11}$ for $bs=256$, same value for all methods}
\label{ris:model_problem_perfromance}
\end{figure*}

In the shared memory regime one cannot see how communication affects the performance of the algorithms. This becomes clear in the next test case.

\paragraph{Performance on multiple nodes}

In this benchmark we keep the computational load per node approximately constant, i.e. when doubling the matrix size, we also double the number of nodes. The system size is $n = 4 \times 10^4 \times n_{nodes},$ $\alpha = 0.005, $ task size is $2048.$ The number of nodes ranges from 1 to 64, the target accuracy is $\sigma = 10^{-6}.$ The leaf blocksize $bs$ is set to 64. We perform the tests twice and then use the average values for visualizations.    

The results are presented in Figure \ref{ris:model_problem_perfromance_multinode_ta1e-6}, left panel. As one can observe, when communication is involved, the performance growth for Truncmul with increasing number of nodes is close to linear (subplot). In terms of timing, similarly to the single node case, Truncmul is noticeable slower, then the other two (main plot). The Hybrid method gives the best timings. The curves on the main plot are not linear, but rather logarithmic. It agrees with \citep{rubensson2016locality}, where it is shown that the Chunks and Tasks implementation of sparse matrix multiplication has communication costs proportional to $\log_2(n).$ If communication is for free, then one can expect to arrive at a flat region on that plot if the matrix size is increased further, as the problem size grows simultaneously with processing power. However, communication costs cannot be excluded and therefore we observe a logarithmic behavior.   

On the right panel of Figure \ref{ris:model_problem_perfromance_multinode_ta1e-6} we present how different the numbers of calls made to the \verb|gemm| routine are. The first observation is that Truncmul generates significantly more calls than its competitors. The second observations is that the Hybrid method does generate fewer \verb|gemm| calls in this setting. At average, Truncmul generates 68 \% more calls than \textit{SpAMM} does and 74 \% more than Hybrid does. We also notice that in order to preserve accuracy, one has to decrease the truncation threshold value when the matrix size grows. For instance, going from size $4 \times 10^4$ to $256 \times 10^4$ requires to decrease $\tau$ by one order of magnitude. This is due to the hidden $\sqrt{n}$ factor inside the Frobenius norm, which is used in the methods. We also can observe that Truncmul requires smaller values of threshold compared to the other two methods. For example, for $n = 256 \times 10^4$, both \textit{SpAMM} and Hybrid require $\tau$ to be $10^{-11}$, whereas Truncmul is able to calculate the result accurately enough only with $\tau = 10^{-12}.$

\begin{figure*}[ht]
\begin{minipage}[h]{0.45\linewidth}
\center{\includegraphics[width=1\linewidth]{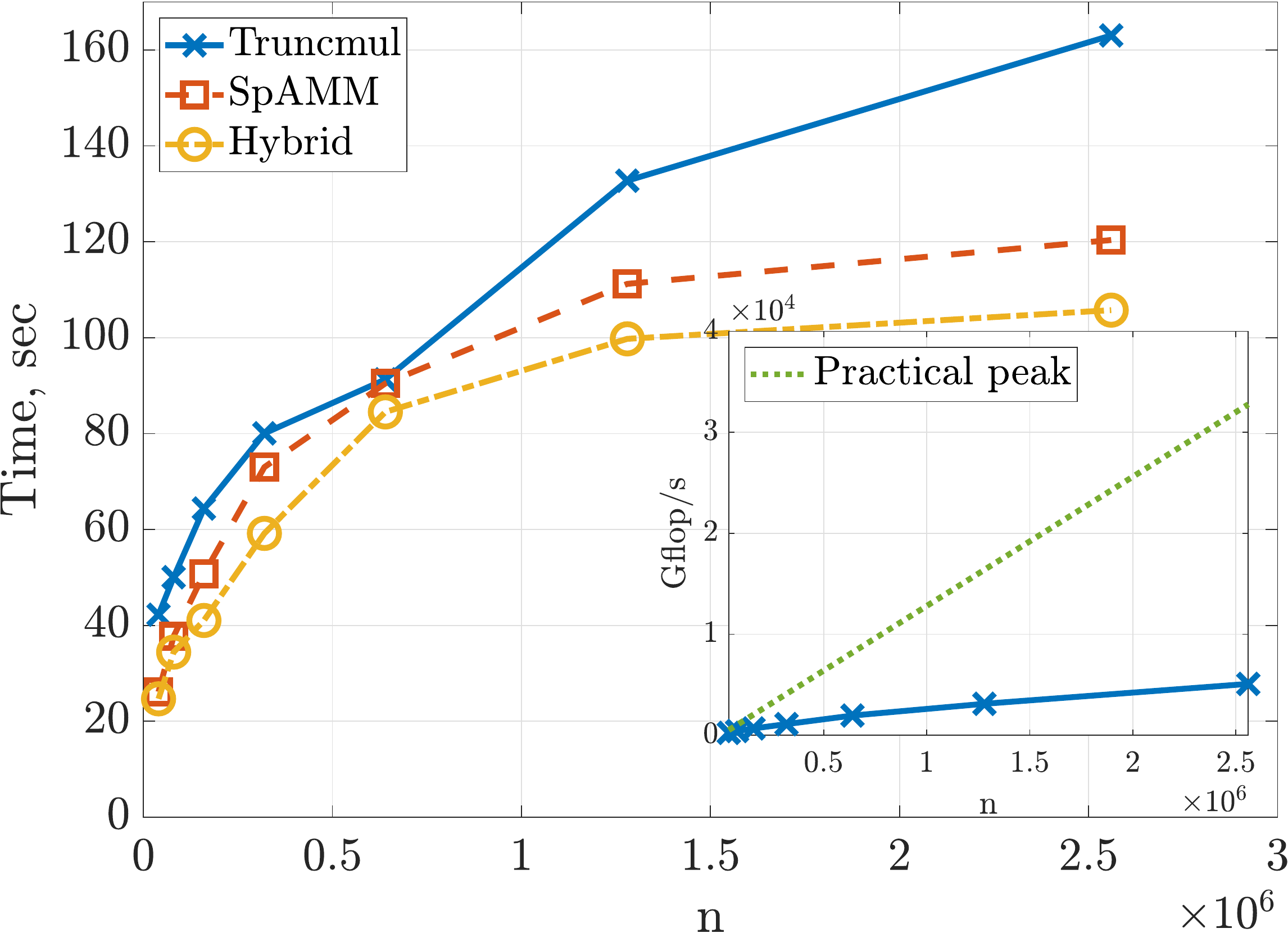}} 
\end{minipage}
\hfill
\begin{minipage}[h]{0.45\linewidth}
\center{\includegraphics[width=1\linewidth]{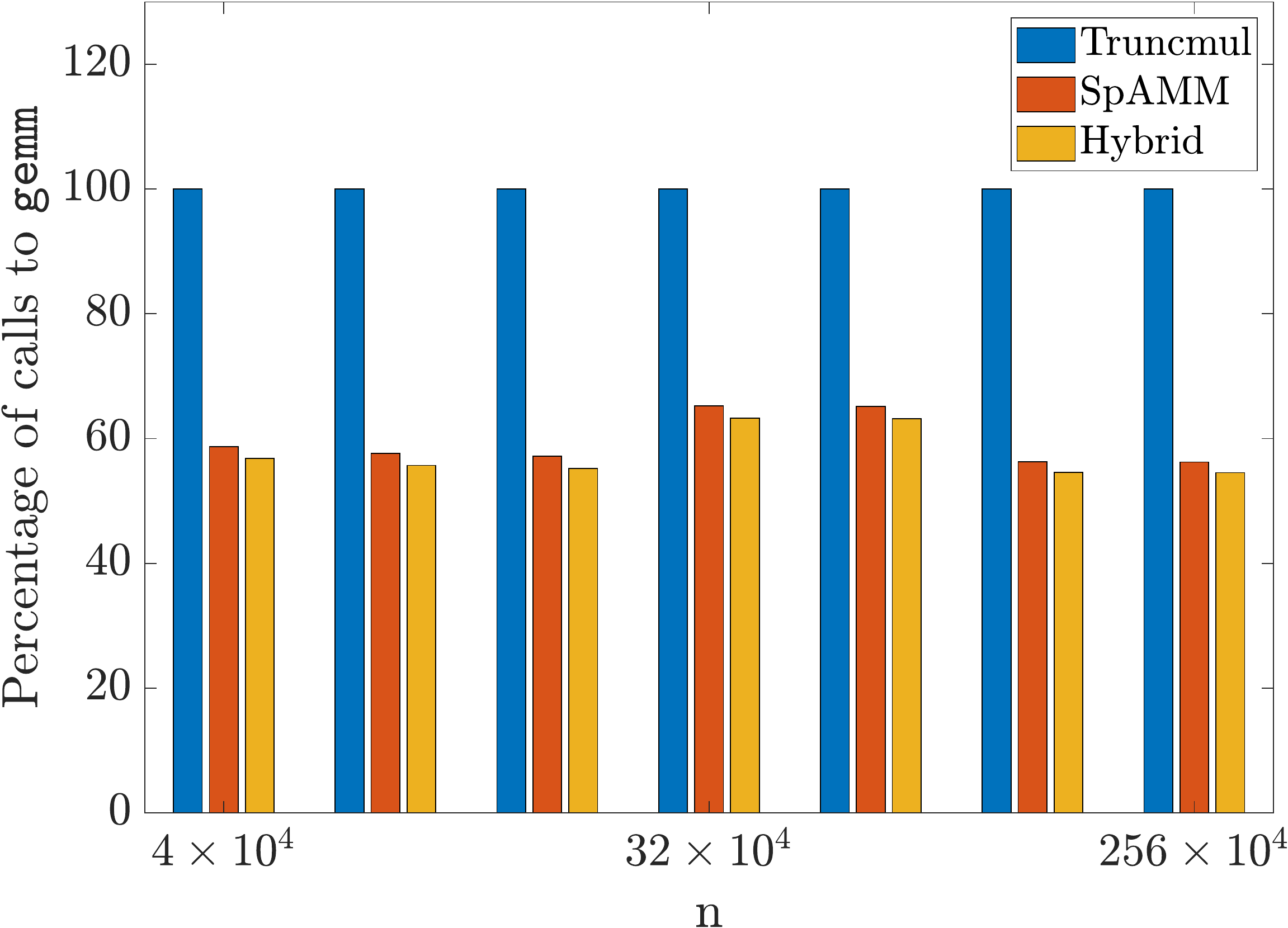}} 
\end{minipage}
\caption{Left panel: multiple node timings on the model problem for the three methods (main plot), performance of Truncmul (subplot), target accuracy $\sigma = 10^{-6}$. Right panel: percentage of the number of calls to gemm BLAS matrix multiply routine compared to Truncmul. Computational load is kept approximately the same by doubling the matrix size simultaneously with the number of nodes, $n = 4 \times 10^4 \times n_{nodes}$, block-size is 64. For $n = 4 \times 10^4$ truncation is performed with $\tau = 10^{-11}$ for Truncmul and $\tau = 10^{-10}$ for the other methods. For $n = 256 \times 10^4$ truncation is performed with $\tau = 10^{-12}$ for Truncmul and $\tau = 10^{-11}$ for the other methods }
\label{ris:model_problem_perfromance_multinode_ta1e-6}
\end{figure*}

To verify the error estimates derived in Section \ref{error_control} we plot the experimental data in Figure \ref{ris:model_problem_multinode_errors}. The error norm for all three algorithms behaves the same way: it grows proportionally to the square root of the system size when the system size grows and decreases almost proportionally to $\tau$ when this parameter is decreased. 

\begin{figure*}[ht]
\begin{minipage}[h]{0.45\linewidth}
\center{\includegraphics[width=1\linewidth]{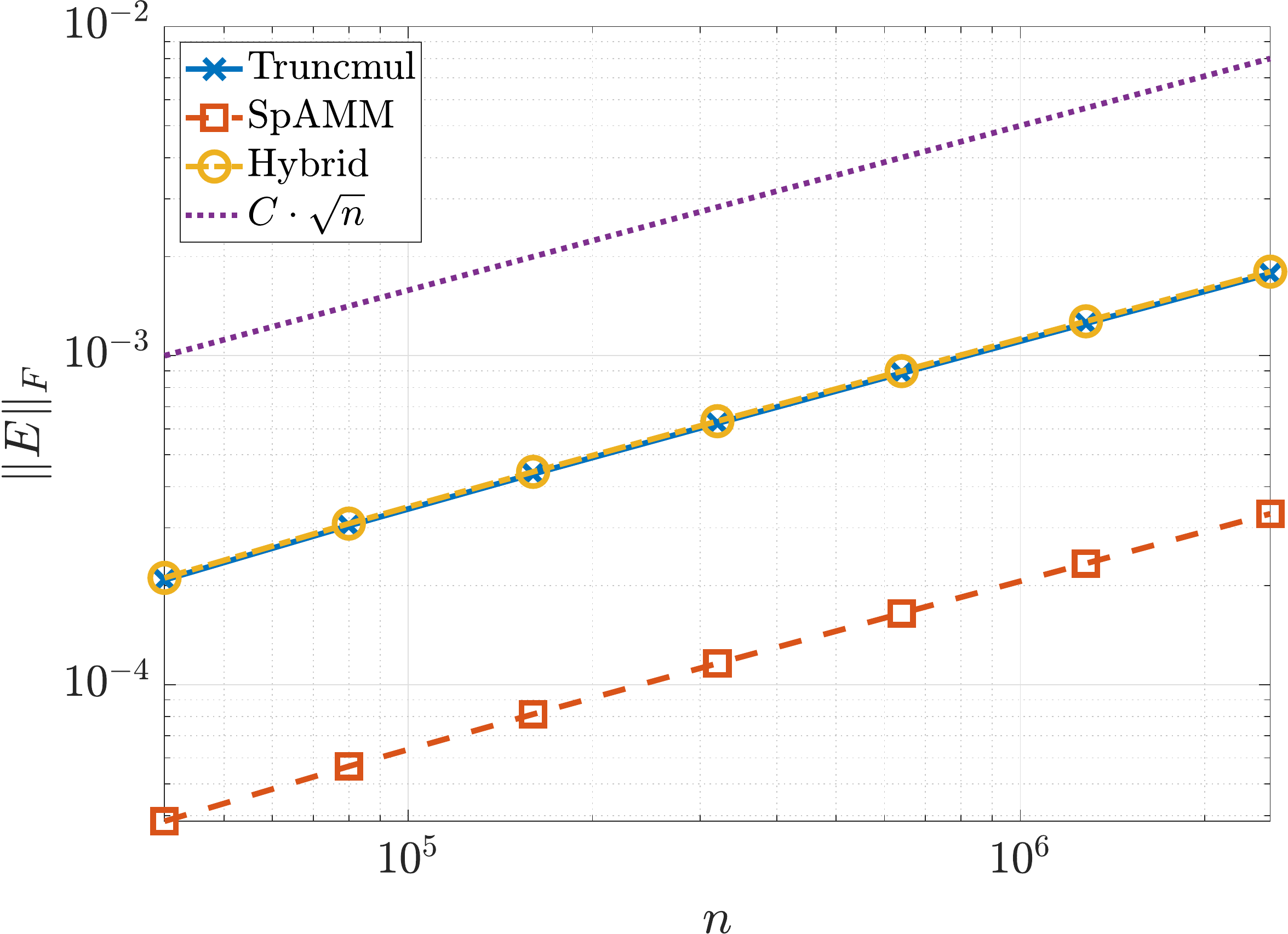}}
\end{minipage}
\hfill
\begin{minipage}[h]{0.45\linewidth}
\center{\includegraphics[width=1\linewidth]{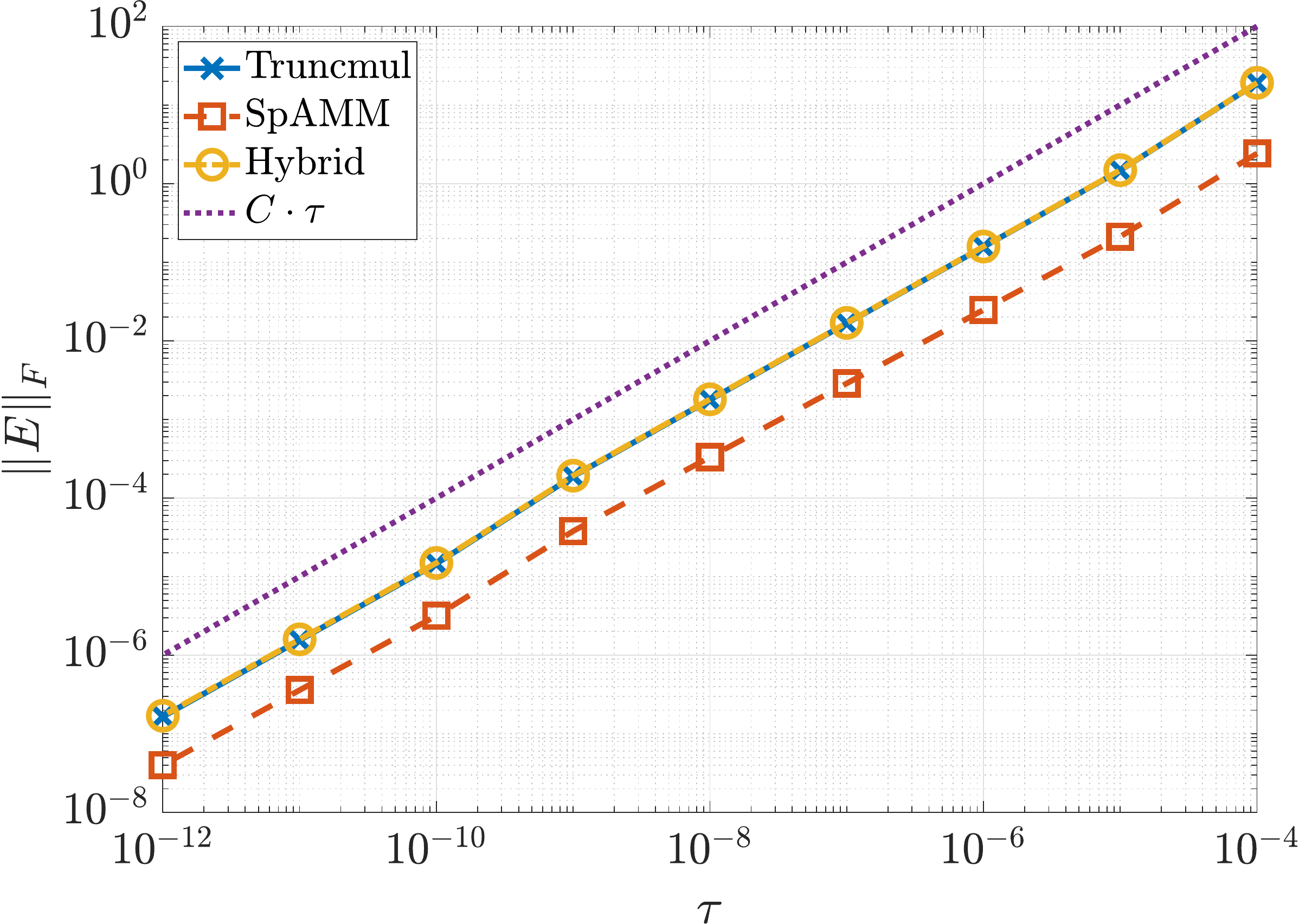}} 
\end{minipage}
\caption{Errors in the model problem. Left panel: the Frobenius norm of the error matrix as a function of $n$ for fixed $\tau = 10^{-6}.$ All lines have approximate slope of 0.5. Right panel: the Frobenius norm of the error matrix as a function of $\tau$ parameter for $n = 256 \times 10^4$. All three lines have approximate slope of one. Dotted lines are for references. }
\label{ris:model_problem_multinode_errors}
\end{figure*}

\subsection{Performance on a real problem} 

In this benchmark we use so-called basis set overlap matrices originating from space discretization of water clusters of increasing size using a Gaussian basis set. A water cluster is a set of water molecules of bulk water clustered inside a sphere of certain radius at standard conditions. The process of the geometry generation is described in \cite{rudberg2011kohn}. The overlap matrices obey the exponential decay property with respect to Euclidean distance between the centres of Gaussian basis functions, which are placed at nuclei positions. The corresponding xyz coordinates can be downloaded from \url{http://ergoscf.org/}. The same coordinate files are also used in \cite{Artemov2018parallelization}. The software used to generate the matrices is the Ergo open-source program for linear-scaling electronic structure calculations \cite{Ergo-SoftwareX-2018}, distributed under the GNU Public Licence v3 and freely obtainable at \url{http://ergoscf.org/}. The basis set used for this particular set of experiments is STO-3G.

\paragraph{Approximate weak scaling} This computational experiment is performed as follows: an overlap matrix $S$ is computed using the Ergo code for the given .xyz file and its square is used as a reference to measure errors. Similarly to the previous tests, we fix a target accuracy $\sigma = 10^{-6}$ and then the matrix is multiplied by itself not using that it is symmetric by all algorithms for the same value of truncation parameter $\tau$ ranging from $10^{-4}$ to $10^{-12}.$ The Frobenius norms of the errors in the product matrices are measured and for each algorithm we pick the largest value of the threshold such that the error norm does not exceed the target accuracy $\sigma$. We collect the statistics and do comparison of the runs made with corresponding threshold values.

The system size is increased by a factor of $\approx 2$ every time as well as the number of worker processes involved, and each process handles about 6700 atoms, i.e. the test shows an approximate weak scaling. We set 31 threads per process to perform tasks and left a single thread for communication, so a single process occupies the whole node. The matrix is split into blocks of maximum size 2048 (task size) and those blocks are handled by the leaf level library, which uses block-size 64. We perform four runs of each test. 

The results for the particular value of target accuracy $\sigma = 10^{-6}$ are presented in Figure \ref{ris:waterclusters_weak_scaling_plus_data_sent}, the left panel. As one can observe, the $\textit{SpAMM}$ algorithm cannot compete with the others almost for any sizes. We would like to point out that all three algorithms inherit logarithm-like weak scaling curve, since they all are based on the locality-aware matrix multiplication \citep{rubensson2016locality}.

\paragraph{Data movement} The main reason why the $\textit{SpAMM}$ algorithm behaves badly on a real problem is the amount of data to be moved. Since no initial truncation is applied, it has to manipulate very many small sub-matrices. The other two algorithms do employ the truncation step before starting the main procedure, so they do not have to move so much data. We collect statistics after four sets of tests and the plots representing the amount of data sent by a worker are presented in Figure \ref{ris:waterclusters_weak_scaling_plus_data_sent}, the right panel. There are two main observation one can do when looking at the plots: 1) on average, the $\textit{SpAMM}$ algorithms moves more data than its competitors and 2) the gap between the average amount and the maximal amount of the data sent for the $\textit{SpAMM}$ algorithm is about 260\%. The Truncmul algorithm and the Hybrid algorithm demonstrate better results, keeping both the average and the gap low.

\begin{figure*}[h]
\begin{minipage}[h]{0.45\linewidth}
\center{\includegraphics[width=1\linewidth]{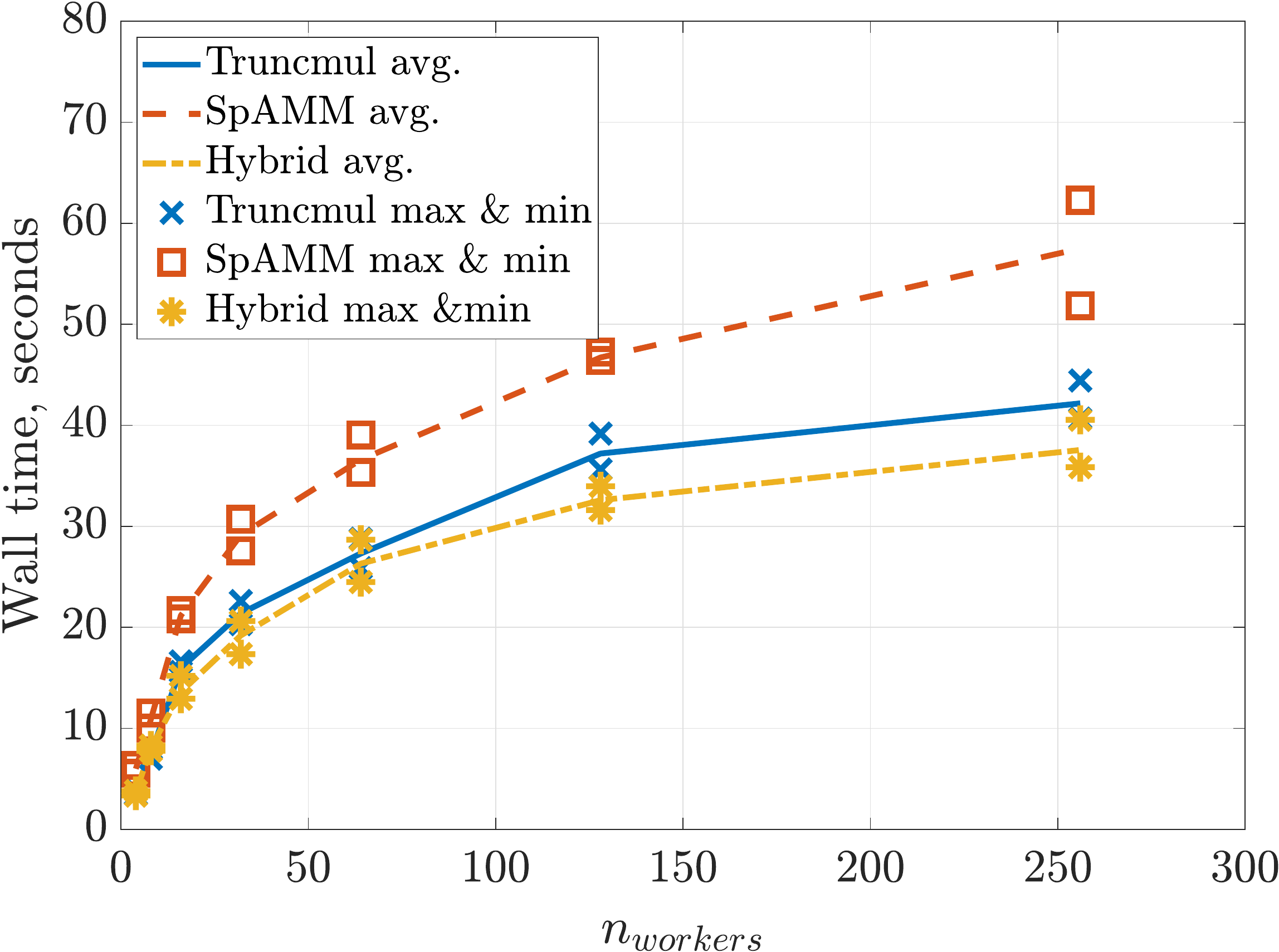}} 
\end{minipage}
\hfill
\begin{minipage}[h]{0.45\linewidth}
\center{\includegraphics[width=1\linewidth]{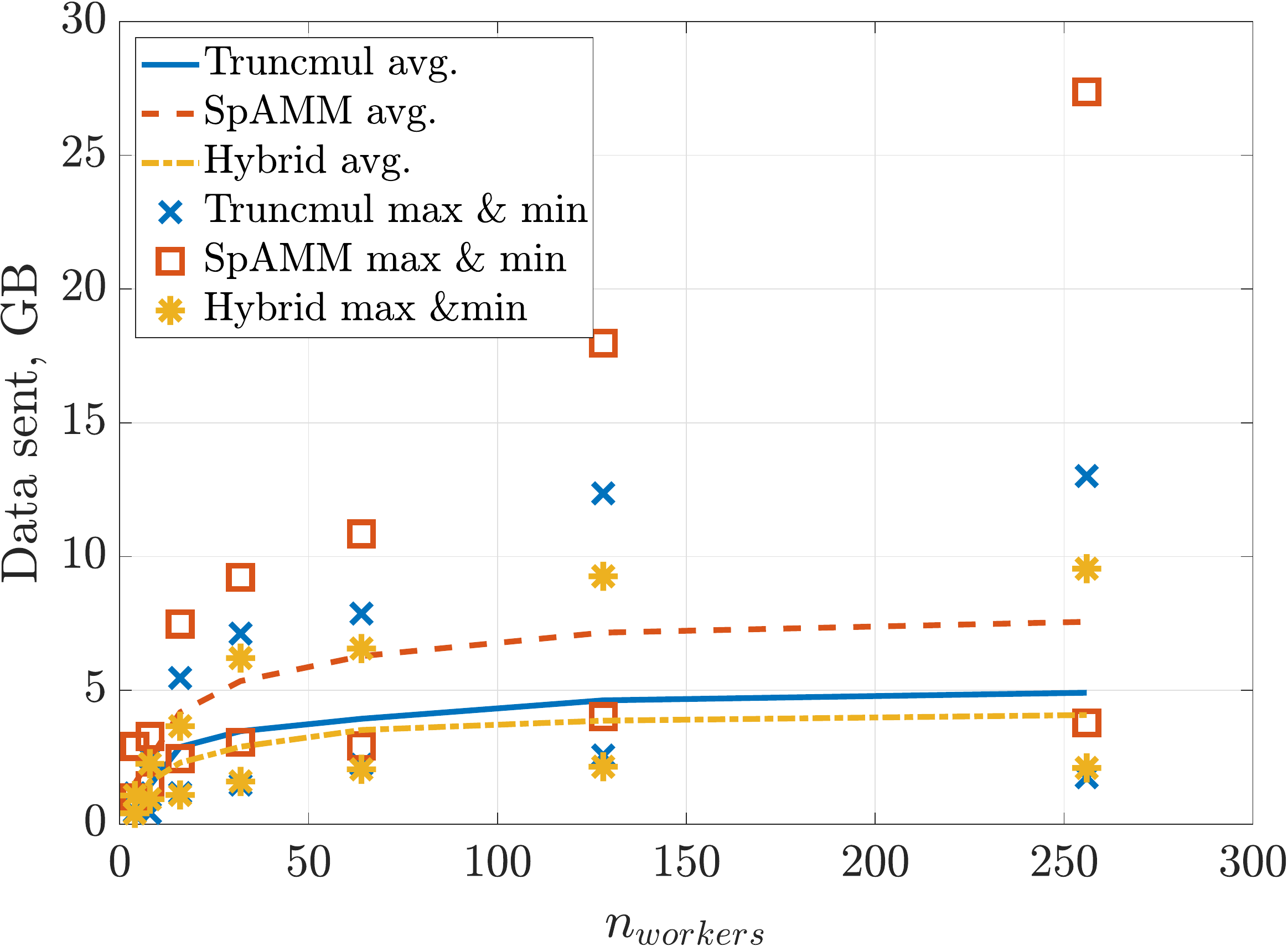}} 
\end{minipage}
\caption{Left panel: approximate weak scaling of the algorithms for water clusters of increasing size, target accuracy $\sigma = 10^{-6}$. The number of atoms per process was approximately fixed to $6.7 \times 10^3$, so that the system size is scaled up together with the number of working processes involved. Lines represent average execution time and points represent maximal and minimal execution time over all the workers. Right panel: average amount of data sent by a worker when executing the corresponding algorithm for water clusters of increasing size, $\sigma = 10^{-6}$. Points represent maximal and minimal amounts over all the workers.}
\label{ris:waterclusters_weak_scaling_plus_data_sent}
\end{figure*}

\paragraph{Errors} In order to verify the error estimates obtained in Section \ref{error_control}, we plot the Frobenius norms of the error matrices as functions of the system size $n$ (or, equivalently, as functions of the number of worker processes, since we are in a weak scaling setting), and then we fix the number of processes and plot the norms as functions of $\tau.$ The corresponding plots are presented in Figure \ref{ris:waterclusters_weak_scaling_errors}. The left panel shows the case when the $\tau$ parameter is fixed. In this case, the error grows as the square root of $n$, since the slopes of all the three lines are close to $0.5.$ The right panel demonstrates the second setting, then the systems size is fixed, but $\tau$ varies. In this case, one can observe that the error matrix norm depends almost linearly on $\tau$, which agrees with Theorems \ref{theorem_truncmul_new}, \ref{theorem_spamm_new} and \ref{theorem_tuncated_spamm_new}. The slopes of the lines are $\approx 0.99$ for the multiplication of truncated matrices, $\approx 0.96$ for the $\textit{SpAMM}$ algorithm and $\approx 0.97$ for the Hybrid approach. In practice, for a sufficiently large matrix, the dependency between the error matrix norm and the $\tau$ parameter becomes virtually linear.

\begin{figure*}[h]
\begin{minipage}[h]{0.45\linewidth}
\center{\includegraphics[width=0.91\linewidth]{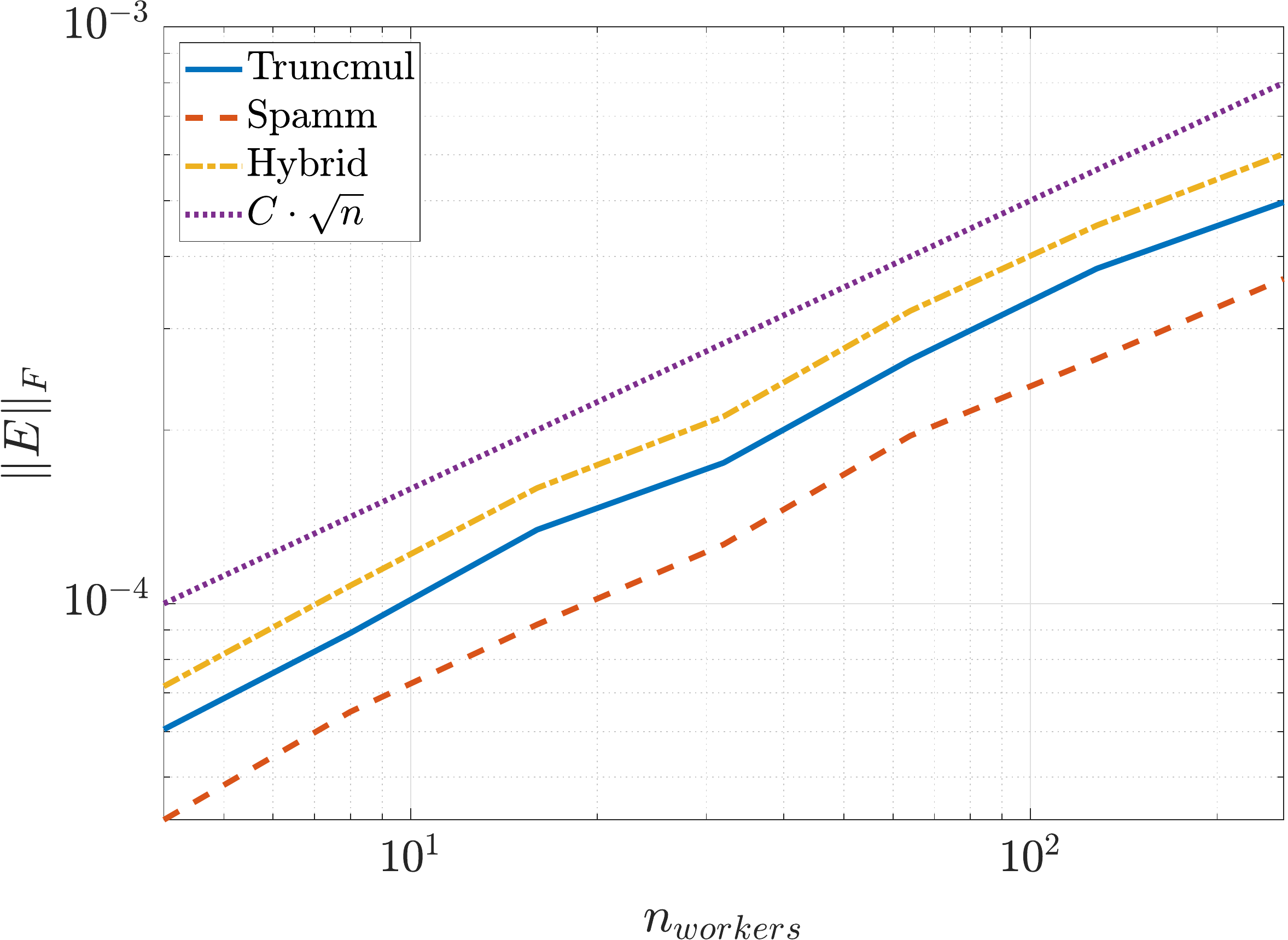}} 
\end{minipage}
\hfill
\begin{minipage}[h]{0.45\linewidth}
\center{\includegraphics[width=0.91\linewidth]{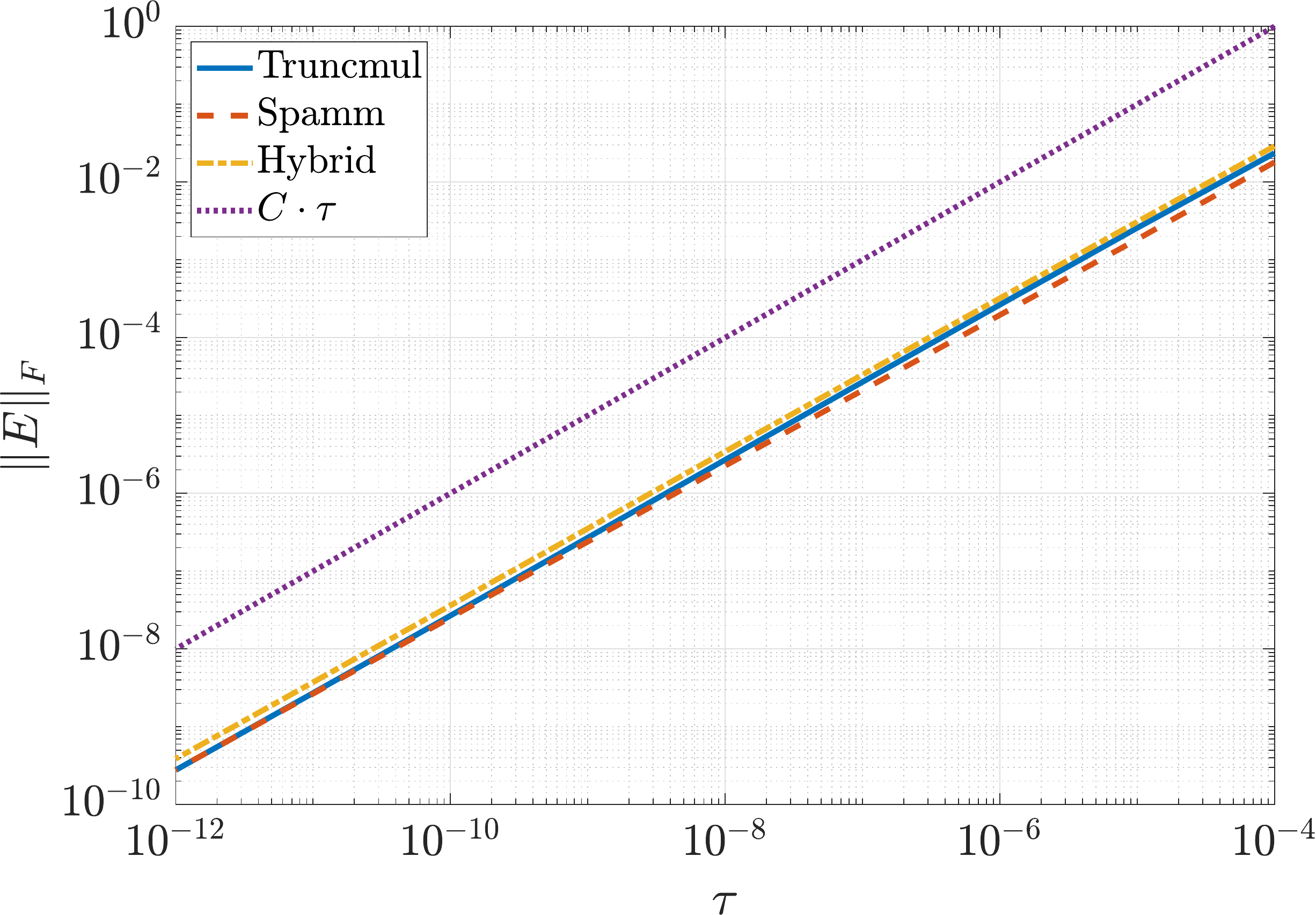}} 
\end{minipage}
\caption{Errors in the real problem. Left panel: the Frobenius norms of error matrices as functions of $n$ or $n_{workers}$ for fixed truncation threshold $\tau = 10^{-6}$. The number of atoms per process is approximately fixed to $6.7 \times 10^3$ so that the system size is scaled up together with the number of working processes involved. The slope of the dotted line (reference) is 0.5. Right panel: the Frobenius norms of error matrices as functions of truncation threshold $\tau$ for fixed $n_{workers} = 64$ or $n = 432498$ atoms, which corresponds to matrix size 1009162. The slopes are $\approx 0.99$ for multiplication of truncated matrices,  $\approx 0.96$ for $\textit{SpAMM}$ and $\approx 0.97$ for the Hybrid method, the dotted line is for reference.}
\label{ris:waterclusters_weak_scaling_errors}
\end{figure*}

\section{Discussion and concluding remarks} \label{conclusion} 

We developed three implementations of approximate multiplication of sparse matrices with exponential decay using the Chunks and Tasks programming model. The considered methods are multiplication of truncated matrices, the $\textit{SpAMM}$ algorithm and their combination, referred to as 'hybrid'. We also developed a new leaf level library, which exploits hierarchical block-sparse representation and is competitive with the previously used block-sparse library in terms of performance, especially for matrices with low fill-in. 

We also derived the asymptotic behavior of the absolute error norm and demonstrated that if $\textit{SpAMM}$ is applied to truncated matrices with exponential decay, then this extra truncation step does not qualitatively change the behavior of the Frobenius norm of the error matrix. When testing on matrices coming from quantum chemistry problems, a drawback of the $\textit{SpAMM}$ algorithm shows up. It performs poorly due to large amounts of data to be moved. The proposed hybrid approach reduces the communication significantly and performs better.

A strong point of our implementations is that it utilizes a work-stealing approach for load balancing and does not require to scatter matrices over the processes using random permutations. The permutation approach is used, for instance, in the NTpoly code \cite{dawson2018massively}, as well as in the SpSUMMA method \cite{buluc2008challenges}, \cite{bulucc2012parallel}. If matrices are permuted, then their localization structure is lost and cannot be exploited. As it is shown in \citep{rubensson2016locality}, multiplication routine exploiting locality features and utilizing work-stealing approach to load balancing demonstrates superior weak scaling performance w.r.t. the SpSUMMA algorithm. As for the 3D multiplication method, asymptotically it has the same behavior as our implementation, but in practice one has to go to much larger matrices to see that. Our implementation also does not rely on collective communications, which require synchronization and affect performance and might create some other problems such as deadlocks, whereas it is heavily used in the mentioned approaches.

We conclude that the $\textit{SpAMM}$ algorithm for multiplication of nearly sparse matrices with exponential decay with respect to a distance function can be successfully used in a shared memory environment, but it is not the best choice for a distributed environment. Adding a truncation step before application of $\textit{SpAMM}$ improves the behavior of the algorithm drastically. This hybrid approach does have the same error behavior as the original $\textit{SpAMM}$ algorithm and the multiplication of truncated matrices, but performs better due to reduced communication.

\section{Acknowledgement}
Support from the Swedish national strategic e-science research program
(eSSENCE) is gratefully acknowledged. Computational resources were
provided by the Swedish National Infrastructure for Computing (SNIC)
at the PDC Center for High Performance Computing at the KTH Royal
Institute of Technology in Stockholm.

We thank Assoc. Prof. Emanuel H. Rubensson and Prof. Maya Neytcheva for their fruitful comments that helped to improve the manuscript.

\appendix
\section{Proofs of Lemmas and Theorems}\label{appendix_with_proofs}

\begin{proof}[Proof of Lemma \ref{lemma_sum_insignificant}]

One should consider two cases, in which insignificant elements can exist: $d(i,j) < r_{\varepsilon} = \frac{1}{\alpha} \ln \frac{c}{\varepsilon} $ and $d(i,j) \geq r_{\varepsilon}.$ This is due to inequality sign in \eqref{eq:exp_decay_single_sequence}. Note that it is implicitly assumed that $\varepsilon \leq c.$

Let us first consider the latter case, i.e. $d_n(i,j) \geq r_{\varepsilon}.$ We directly use the exponential decay property \eqref{eq:exp_decay_single_sequence} here. Then
\begin{align*} \nonumber
\sum\limits_{i,j:\, d_n(i,j) \geq r_{\varepsilon}} & | [C_n]_{i,j} |^2 \leq c^2 \sum\limits_{i,j:\, d_n(i,j) \geq r_{\varepsilon}}  e^{-2\alpha d_n(i,j)} \\ &= 
c^2 \sum\limits_{j}\sum\limits_{i:\,\, d_n(i,j) \geq r_{\varepsilon}} e^{-2 \alpha d_n(i,j)} \\
& \leq c^2 \sum\limits_{j} \sum\limits_{r = r_{\varepsilon} + 1}^{\infty} \left( | N_{d_n}(j,r) | - |N_{d_n}(j,r-1)| \right) e^{-2 \alpha (r-1)} \\
& \leq c^2 \sum\limits_{j} \sum\limits_{r = r_{\varepsilon} + 1}^{\infty} | N_{d_n}(j,r) |  e^{-2 \alpha (r-1)} \\
& \leq c^2 \sum\limits_{j} \sum\limits_{r = r_{\varepsilon} + 1}^{\infty} \gamma r^{\beta} e^{-2 \alpha (r-1)} \leq c^2 n \sum\limits_{r = r_{\varepsilon} + 1}^{\infty} \gamma r^{\beta} e^{-2 \alpha (r-1)} .
\end{align*}

One can set $p = r - r_{\varepsilon} - 1,$ then 
\begin{equation} \label{series_insignificant}
\begin{split}
c^2 n \sum\limits_{r = r_{\varepsilon} + 1}^{\infty} \gamma r^{\beta} e^{-2 \alpha (r-1)} &= c^2 n \sum\limits_{p = 0}^{\infty} \gamma (p + r_{\varepsilon} + 1)^{\beta} e^{-2 \alpha (p + r_{\varepsilon})} \\ & = c^2 n e^{-2 \alpha r_{\varepsilon}} \sum\limits_{p = 0}^{\infty} \gamma (p + r_{\varepsilon} + 1)^{\beta} e^{-2 \alpha p}
\end{split} 
\end{equation}

The series in \eqref{series_insignificant} converges due to d'Alembert test \citep{rudin1976principles}. Let $\nu$ be the sum of that series. It is a constant independent of $n$. Note that $e^{-2\alpha r_{\varepsilon}} = e^{-2 \alpha \frac{1}{\alpha} \ln \frac{c}{\varepsilon}} = \frac{\varepsilon^2}{c^2},$ thus
\begin{equation} \nonumber
\sum\limits_{i,j:\, d_n(i,j) \geq r_{\varepsilon}} | [C_n]_{i,j} |^2 \leq n \varepsilon^2 \nu.
\end{equation} 

Since $\nu$ is independent of $n$, $n = O(n)$ and $\lim\limits_{\varepsilon \rightarrow 0} \frac{\varepsilon^2}{\varepsilon^2} = 1$, then
\begin{equation} \label{sum_insignificant_outside_re}
\sum\limits_{i,j:\, d_n(i,j) \geq r_{\varepsilon}} | [C_n]_{i,j} |^2 = \begin{cases} O(n), & n \longrightarrow \infty, \\ O(\varepsilon^2), & \varepsilon \longrightarrow 0. \end{cases} 
\end{equation}

Now one has to treat the other part, i.e. the elements, for which the distance functions is less than $r_{\varepsilon},$ using the fact that the elements are smaller than $
\varepsilon$ in magnitude.  This can be done as follows:
\begin{equation} \nonumber
\begin{split}
\sum\limits_{\substack{i,j:~ d_n(i,j) < r_{\varepsilon} \\ | [C_n]_{i,j} | \leq \varepsilon}} | [C_n]_{i,j} |^2 & \leq \sum\limits_{i,j:~ d_n(i,j) \in [0, r_{\varepsilon})} \varepsilon^2 = \varepsilon^2 \sum\limits_j \sum\limits_{i:~ d_n(i,j) \in [0, r_{\varepsilon})} 1 \\
& = \varepsilon^2 \sum\limits_j |N_{d_n}(j,r_{\varepsilon})|  \leq \varepsilon^2 \sum\limits_{j} \gamma \left( r_{\varepsilon} \right)^{\beta} \\
& \leq \varepsilon^2 n \gamma \left( r_{\varepsilon} \right)^{\beta} = \varepsilon^2 n \gamma \left(\frac{1}{\alpha} \ln\frac{c}{\varepsilon} \right)^{\beta}.
\end{split}
\end{equation}

Since $\alpha, \beta, \gamma$ and $c$ are constants independent of $n$, $o(x) = O(x)$ for any $x$ and 
\begin{equation} \nonumber
\lim_{\varepsilon \rightarrow 0} \frac{\varepsilon^2 \gamma \left(\frac{1}{\alpha} \ln\frac{c}{\varepsilon} \right)^{\beta}}{\varepsilon^p} = 0,\, \forall p < 2,
\end{equation} one concludes that 
\begin{equation} \label{sum_insignificant_within_re}
\sum\limits_{\substack{i,j:~ d_n(i,j) < r_{\varepsilon} \\ | [C_n]_{i,j} | \leq \varepsilon}} | [C_n]_{i,j} |^2  = \begin{cases} O(n), & n \longrightarrow \infty, \\ O(\varepsilon^p),\, \forall p < 2, & \varepsilon \longrightarrow 0. \end{cases}
\end{equation}

Combining together \eqref{sum_insignificant_outside_re} and \eqref{sum_insignificant_within_re}, we obtain 
\begin{equation} \nonumber
\sum\limits_{i,j:\, | [C_n]_{i,j} | \leq \varepsilon} | [C_n]_{i,j} |^2 = \begin{cases} O(n), & n \longrightarrow \infty, \\ O(\varepsilon^p),\, \forall p < 2, & \varepsilon \longrightarrow 0. \end{cases}
\end{equation}
\end{proof}

\begin{proof}[Proof of Theorem \ref{theorem_truncmul_new}]

Any element of a matrix product is computed as follows:
\begin{equation}
[C_n]_{i,j} = \sum\limits_{k=1}^{n} [A_n]_{i,k}[B_n]_{k,j}, \label{scalar_product}
\end{equation} which is a scalar product of the $i$-th row of $A_n$ and the $j$-th column of $B_n$. At the same time, any element of the product of the truncated matrices $\tilde{A}_n$ and $\tilde{B}_n$ is computed as follows:
\begin{equation} \nonumber
[\tilde{C}_n]_{i,j} = \sum\limits_{k=1}^{n}[\phi_n]_{ikj}, 
\end{equation}
\begin{equation} \nonumber
[\phi_n]_{ikj} = \begin{cases} 0, &\text{if } |[A_n]_{i,k}| < \tau\,\, \text{or} \,\, |[B_n]_{k,j}| < \tau , \\ [A_n]_{i,k}[B_n]_{k,j} &\text{otherwise}. \end{cases}
\end{equation} Then any element of the error matrix $E_n = C_n - \tilde{C}_n$ has an absolute value 
\begin{equation} \nonumber
|[E_n]_{i,j}| = \left|\sum\limits_{k=1}^{n} \left( [A_n]_{i,k}[B_n]_{k,j} - [\phi_n]_{ikj} \right)\right|,\,\,\forall i,j = 1,\ldots,n.
\end{equation} This quantity attains zero if no information is lost when computing a given element, i.e. no elements are truncated to zero, and attains its maximum when all information is lost when performing multiplication, i.e. $[\phi_n]_{ikj} = 0,\,\,k = 1,\ldots,n$. The worst-case scenario is when every element in the error matrix attains its maximum, i.e. $\tilde{C}_n = 0$ and thus $E_n = C_n.$

Let us consider the matrix $C_n = A_n B_n$. We have
\begin{equation} \label{c_trunc_mul_element_abs_value}
|[C_n]_{i,j}| = \left| \sum\limits_{k=1}^{n} [A_n]_{i,k}[B_n]_{k,j} \right| \leq \sum\limits_{k=1}^{n}|[A_n]_{i,k}||[B_n]_{k,j}|. 
\end{equation}

Since both $A_n$ and $B_n$ possess on the exponential decay property with respect to a common distance function $d_n(i,j),$ by Theorem 4 in \citep{Rubensson2018localized} for any $\varepsilon > 0$ both matrices have at most $\kappa$ significant elements (i.e. greater than $\varepsilon$ in magnitude) in every row or column, where $\kappa$ is a constant independent of $n$. Let us use $\varepsilon = \tau$. 

When estimating the absolute value in \eqref{c_trunc_mul_element_abs_value}, three possible types of summands are to be taken into account:

\begin{enumerate} 
\itemsep0.5em
\item  $|[A_n]_{i,k}| \geq \tau,\,\, |[B_n]_{k,j}| < \tau \Rightarrow 0 \leq |[A_n]_{i,k}| |[B_n]_{k,j}| \leq c \tau;$

\item $|[A_n]_{i,k}| < \tau,\,\, |[B_n]_{k,j}| \geq \tau \Rightarrow 0 \leq |[A_n]_{i,k}| |[B_n]_{k,j}| \leq c \tau;$

\item $|[A_n]_{i,k}| < \tau,\,\, |[B_n]_{k,j}| < \tau \Rightarrow 0 \leq |[A_n]_{i,k}| |[B_n]_{k,j}| < \tau^2.$
\end{enumerate} The case where both multipliers are greater than $\tau$ in magnitude is not included since it falls outside the worst-case scenario. The estimation of the product of the absolute values for the first two types is due to exponential decay property, $c = \max(c_1, c_2).$ There are not more than $2\kappa$ summands of type 1 and 2 for any element (Theorem 4 from \citep{Rubensson2018localized}). The rest summands are of type 3. We split the index set $I_n = \{1,\ldots,n\}$ into two sets, ${I_n}_2 = \{k: |[A_n]_{i,k}| |[B_n]_{i,k}| ~\text{are of type 3} \}$ and ${I_n}_1 = I_n \setminus {I_n}_2.$ 

Let $P$ be a permutation of the index set such that $P({I_n}_1) =  \{1,\ldots, 2\kappa\}$ and $P({I_n}_2) = \{2\kappa + 1,\ldots,n\}$. Let us denote $P(i) = \bar{i},\,\forall i = 1,\ldots, n.$ The permutation $P$ transforms the matrices $A_n$ and $B_n$ accordingly: $[\bar{A}_n]_{\bar{i},\bar{j}} = [A_n]_{i,j}, [\bar{B}_n]_{\bar{i},\bar{j}} = [B_n]_{i,j},\,\forall i,j = 1,\ldots,n.$ These matrices $\bar{A}_n$ and $\bar{B}_n$ preserve the decay properties, but with respect to another distance function $\bar{d}_n(\cdot,\cdot)$ such that $\bar{d}_n(\bar{i},\bar{j}) = d_n(i,j),\,\forall i,j = 1,\ldots, n.$ Then

\begin{align*} \nonumber
|[E_n]|_{i,j}  = |[C_n]|_{i,j} & \leq \sum\limits_{k \in {I_n}_1} |[A_n]_{i,k}| |[B_n]_{k,j}| + \sum\limits_{k \in {I_n}_2} |[A_n]_{i,k}| |[B_n]_{k,j}| \\
& = \underbrace{\sum\limits_{\bar{k} = 1}^{2\kappa} |[\bar{A}_n]_{\bar{i},\bar{k}}| |[\bar{B}_n]_{\bar{k},\bar{j}}|}_{S_1} + \underbrace{\sum\limits_{\bar{k} = 2\kappa+1}^{n} |[\bar{A}_n]_{\bar{i},\bar{k}}| |[\bar{B}_n]_{\bar{k},\bar{j}}|}_{S_2} \\
& = S_1 + S_2.
\end{align*} We consider next the partial sums $S_1$ and $S_2$ separately, starting with $S_1:$
\begin{equation} \label{S_1_truncmul_done}
S_1 = \sum\limits_{\bar{k} = 1}^{2\kappa} |[\bar{A}_n]_{\bar{i},\bar{k}}| |[\bar{B}_n]_{\bar{k},\bar{j}}| \leq 2\kappa c \tau = \delta_1 \tau. 
\end{equation} 

To estimate $S_2$, one should take into account that all summands in this expression are of type 3, i.e. both $|[\bar{A}_n]_{\bar{i},\bar{k}}| < \tau,$ $|[\bar{B}_n]_{\bar{k},\bar{j}}| < \tau$ $\forall \bar{k} = 2\kappa + 1,
\ldots, n$: 
\begin{equation} \label{S_2_trunc_mul_almost_done}
\begin{split}
S_2 & = \sum\limits_{\bar{k} = 2\kappa+1}^{n} \underbrace{|[\bar{A}_n]_{\bar{i},\bar{k}}|}_{< \tau} |[\bar{B}_n]_{\bar{k},\bar{j}}| < \tau \sum\limits_{\bar{k} = 2\kappa+1}^{n} |[\bar{B}_n]_{\bar{k},\bar{j}}|  \\ 
&\leq c_2 \tau \sum\limits_{\bar{k} = 2\kappa+1}^{n} e^{-\alpha  \bar{d}_n(\bar{k},\bar{j})}.
\end{split} 
\end{equation}

It can be shown that the last sum in \eqref{S_2_trunc_mul_almost_done} is bounded by a constant independent of $n:$
\begin{equation} \nonumber
\begin{split}
\sum\limits_{\bar{k} = 2\kappa+1}^{n} e^{-\alpha \bar{d}_n(\bar{k},\bar{j})} & \leq \sum\limits_{k = 1}^{n} e^{-\alpha \bar{d}_n(\bar{k},\bar{j})} \\ 
& \leq \sum\limits_{r=1}^{\infty}\left(|N_{\bar{d}_n}(\bar{j},r)| - |N_{\bar{d}_n}(\bar{j},r-1)| \right)e^{-\alpha(r-1)} \\
& \leq \sum\limits_{r=1}^{\infty} \left( |N_{\bar{d}_n}(\bar{j},r)| \right)e^{-\alpha(r-1)} \leq \sum\limits_{r=1}^{\infty} \gamma r ^{\beta}e^{-\alpha(r-1)},
\end{split}
\end{equation} where the last series is a convergent one, which can be demonstrated with d'Alembert test \citep{rudin1976principles}. Let $\sigma$ denote its sum, then
\begin{equation} \label{S_2_truncmul_done}
S_2 < c_2 \tau \sigma = \delta_2 \tau.
\end{equation} 

By combining \eqref{S_1_truncmul_done} and \eqref{S_2_truncmul_done}, we obtain
\begin{equation} \nonumber
| [E_n]_{i,j} | < \delta_1 \tau + \delta_2 \tau = (\delta_1 + \delta_2) \tau = \delta \tau,\,\,\forall i,j = 1,\ldots,n, \nonumber
\end{equation}

\begin{equation} \nonumber
| [E_n]_{i,j} | = O(\tau) \,\,\forall i,j = 1,\ldots,n.
\end{equation}

In the worst-case scenario, $\tilde{C}_n = 0$ and thus $E_n = C_n.$ We know by Theorem 5 from \citep{Rubensson2018localized} that $C_n$ satisfies the exponential decay property with respect to the same distance function as $A_n$ and $B_n.$ Thus, by Theorem 4 from \citep{Rubensson2018localized} $C_n$ has at most $O(n)$ significant elements $\forall \varepsilon > 0$. In order to estimate the Frobenius norm of the error matrix, we use Lemma \ref{lemma_sum_insignificant} to compute partial sum for all insignificant elements with $\varepsilon = \tau$ and the result derived above to the rest of elements, i.e. significant ones:
\begin{equation} \nonumber
\begin{split}
\| E_n \|^2_F & = \| C_n \|^2_F = \sum\limits_{i,j: |[C_n]_{i,j}| \leq \tau}  |[C_n]_{i,j}|^2 + \sum\limits_{i,j: |[C_n]_{i,j}| > \tau}  |[C_n]_{i,j}|^2 \\
& =  \begin{cases} O(n), & n \longrightarrow \infty, \\ O(\varepsilon^p),\, \forall p < 2, & \varepsilon \longrightarrow 0. \end{cases}
\end{split}
\end{equation} By taking the square root from both sides of the previous relation, we arrive at
\begin{equation} \nonumber
\| E_n \|_F = \begin{cases} O(n^{1/2}), & n \longrightarrow \infty, \\ O(\varepsilon^{p/2}),\, \forall p < 2, & \varepsilon \longrightarrow 0. \end{cases}
\end{equation}
\end{proof}

\begin{proof}[Proof of Theorem \ref{theorem_spamm_new}]
Any element of a matrix product is computed as:
\begin{equation} \nonumber
[C_n]_{i,j} = \sum\limits_{k=1}^{n} [A_n]_{i,k}[B_n]_{k,j}, 
\end{equation} which is nothing but a scalar product of the $i$-th row of $A_n$ and the $j$-th column of $B_n$. At the same time, any element of a $\textit{SpAMM}$ product is computed as follows:
\begin{equation} \nonumber
[\bar{C}_n]_{i,j} = \sum\limits_{k=1}^{n}[\chi_n]_{ikj},\,\,
\end{equation}
\begin{equation} \nonumber
[\chi_n]_{ikj} = \begin{cases} 0, &\text{if } |[A_n]_{i,k}| |[B_n]_{k,j}| < \tau, \\ [A_n]_{i,k}[B_n]_{k,j} &\text{otherwise}. \end{cases}
\end{equation} Then any element of the error matrix $E_n = C_n - \bar{C}_n$ has an absolute value 

\begin{equation} \nonumber
|[E_n]_{i,j}| = \left|\sum\limits_{k=1}^{n} \left( [A_n]_{i,k}[B_n]_{k,j} - [\chi_n]_{ikj} \right)\right|,\,\,\forall i,j = 1,\ldots,n.
\end{equation} This quantity attains zero if no $\textit{SpAMM}$ condition has been met when computing a given element, i.e. the information is preserved, and attains its maximum when all information is lost when performing $\textit{SpAMM}$, i.e. $[\chi_n]_{ikj} = 0,\,\,k = 1,\ldots,n$. The worst-case scenario is when every element in the error matrix attains its maximum, i.e. $\bar{C}_n = 0$ and thus $E_n = C_n.$

Let us consider the matrix $C_n = A_n B_n$. We have
\begin{equation} \label{c_spamm_element_abs_value}
|[C_n]_{i,j}| = \left| \sum\limits_{k=1}^{n} [A_n]_{i,k}[B_n]_{k,j} \right| \leq \sum\limits_{k=1}^{n}|[A_n]_{i,k}||[B_n]_{k,j}|. 
\end{equation}

Since both $A_n$ and $B_n$ have exponential decay with respect to a common distance function $d_n(i,j),$ by Theorem 4 in \citep{Rubensson2018localized} for any $\varepsilon > 0$ both matrices have at most $\kappa$ significant elements (i.e. greater than $\varepsilon$ in magnitude) in every row or column, where $\kappa$ is a constant independent of $n$. Let us use $\varepsilon = \tau,$ where $\tau$ is a $\textit{SpAMM}$ parameter. Note that $|[A_n]_{i,k}| |[B_n]_{k,j}| < \tau\,\,\forall i,k,j=1,\ldots,n$ since we are at the worst case scenario.

When estimating the absolute value in \eqref{c_spamm_element_abs_value}, three possible types of summands are to be taken into account:

\begin{enumerate}
\itemsep0.5em


\item  $|[A_n]_{i,k}| \geq \tau,\,\, |[B_n]_{k,j}| < \tau \Rightarrow 0 \leq |[A_n]_{i,k}| |[B_n]_{k,j}| < \tau;$

\item $|[A_n]_{i,k}| < \tau,\,\, |[B_n]_{k,j}| \geq \tau \Rightarrow 0 \leq |[A_n]_{i,k}| |[B_n]_{k,j}| < \tau;$

\item Both $|[A_n]_{i,k}| < \tau$ and $|[B_n]_{k,j}| < \tau \Rightarrow 0 \leq |[A_n]_{i,k}| |[B_n]_{k,j}| < \tau^2.$
\end{enumerate} There are not more than $2\kappa$ summands of type 1 and 2 for any element due to Theorem 4 from \citep{Rubensson2018localized}. The rest summands are of type 3. We split the index set $I_n = \{1,\ldots,n\}$ to two sets, ${I_n}_2 = \{k: |[A_n]_{i,k}| |[B_n]_{i,k}| ~\text{are of type 3} \}$ and ${I_n}_1 = I_n \setminus {I_n}_2.$ The further proof is technical and very similar to the proof of Theorem \ref{theorem_truncmul_new}, therefore it is omitted. 
\end{proof}

\begin{proof}[Proof of Theorem \ref{theorem_tuncated_spamm_new}]
Let $C_n$ be the true product of $A_n$ and $B_n$ and $\hat{C}_n$ be the result of the $\textit{SpAMM}$ algorithm applied on matrices $A_n$ and $B_n$ truncated with the same $\tau > 0$ as used by $\textit{SpAMM}$. Then, the error matrix $E_n = C_n - \hat{C}_n$ satisfies
\begin{equation} \label{eq:split_error}
\begin{split}
\left|[E_n]_{i,j}\right| & = \left| [C_n]_{i,j} - [\hat{C}_n]_{i,j} \right| = \left| [C_n]_{i,j} - [\tilde{C}_n]_{i,j} + [\tilde{C}_n]_{i,j} - [\hat{C}_n]_{i,j}\right| \\
& \leq \left|  [C_n]_{i,j} - [\tilde{C}_n]_{i,j} \right| + \left| [\tilde{C}_n]_{i,j} - [\hat{C}_n]_{i,j} \right|,\,\,\forall i,j = 1,\ldots,n .
\end{split}
\end{equation}

The first expression in \eqref{eq:split_error} is the element-wise error introduced by truncation of matrices before multiplication, and its bound is derived in Theorem \ref{theorem_truncmul_new}. The second expression in \eqref{eq:split_error} is the error introduced by the $\textit{SpAMM}$ algorithm assuming that the matrices have been already truncated, and its bound is derived in Theorem \ref{theorem_spamm_new}. Combination of those results gives us 
\begin{equation} \nonumber
\left|[E_n]_{i,j}\right| = O(\tau).
\end{equation} The further proof is technical and very similar to the proofs of Theorem \ref{theorem_truncmul_new} and Theorem \ref{theorem_spamm_new}, therefore it is omitted. 
\end{proof}

\bibliography{references}

\end{document}